\documentclass{amsart}
\usepackage[latin1]{inputenc}
\usepackage{amssymb}
\usepackage{amsmath}
\usepackage{amscd}
\usepackage{latexsym}
\usepackage{amsfonts}
\usepackage{verbatim}
\usepackage{enumerate}
\usepackage{amsthm}
\usepackage[alphabetic]{amsrefs}

\newtheorem{theorem}{Theorem}
\newtheorem{cor}{Corollary}
\newtheorem{lemma}{Lemma}
\newtheorem{prop}{Proposition}

\newcommand\ric{\operatorname{Ric}}
\numberwithin{equation}{section}

\begin{document}
\title[Connected sum construction of constant Q-curvature manifolds]{Connected sum construction of constant Q-curvature manifolds in higher dimensions}
\author{Yueh-Ju LIN }
\address{Department of Mathematics, University of Notre Dame, Notre Dame, IN 46556, USA}
\email{ylin4@nd.edu}
\date{}

\begin{abstract}
For a compact Riemannian manifold $(M, g_2)$ with constant $Q$-curvature of dimension $n\geq 6$ satisfying nondegeneracy condition, we show that one can construct many examples of constant $Q$-curvature manifolds by gluing construction. We provide a general procedure of gluing together $(M,g_2)$ with any compact manifold $(N, g_1)$ satisfying a geometric assumption. In particular, we can prove that there exists a metric with constant $Q$-curvature on the connected sum $N\#M$.
\end{abstract}
\maketitle

\section{Introduction}\label{intro}
In conformal geometry, the Yamabe problem asks whether there exists a conformal metric with constant scalar curvature.
Let $M$ be a compact manifold of dimension $n\geq 3$ and $R_{g}$ be the scalar curvature of the metric ${g}$. If $\tilde{g}= \psi^{\frac{4}{n-2}}g$ is conformal to ${g}$, where $\psi$ is a smooth positive function, then $R_{\tilde{g}}$ and $R_{g}$ are related by the following nonlinear PDE:
\begin{equation}\label{eq1}
-\frac{4(n-1)}{(n-2)}\Delta_{g}\psi + R_{g}\psi = R_{\tilde{g}}\psi^{\frac{n+2}{n-2}},
\end{equation}
where $\mathcal{L}_{g}=  -\frac{4(n-1)}{(n-2)}\Delta_{g} + R_{g} $ is called the conformal Laplacian operator and
\begin{equation}\label{eqC}
\mathcal{L}_{\tilde{g}}(u) = \psi^{-\frac{n+2}{n-2}}\mathcal{L}_{g}(u\psi).
\end{equation}
Thus, by equation (\ref{eq1}), the Yamabe problem is equivalent to solving
\begin{equation}
\mathcal{L}_{g}(\psi)= \mu\psi^{\frac{n+2}{n-2}}
\end{equation}
for some constant $\mu$.
The problem was solved through a sequence of results: Yamabe \cite{Yam60}, Trudinger \cite{Tru68}, Aubin \cite{Aub76} and Schoen \cite{Sch84}. More details are in the survey paper written by Lee and Parker \cite{LP87}.

In geometric analysis, gluing constructions are well-known methods to create new solutions to nonlinear PDEs from existing ones. Let $(M',g')$ and $(M'', g'')$ be compact manifolds of dimension $n\geq 3$. Suppose $g'$ and $g''$ have constant scalar curvature (not necessarily the same value). Joyce \cite {Joy03} wrote down explicit $1$-parameter family of metrics $g_{t}$ on the connected sum $M=M'\#M''$ whose scalar curvature is close to constant, and showed that it can be perturbed by a small conformal change to a constant scalar curvature metric. Before Joyce published his thesis work, Mazzeo, Pollack and Uhlenbeck \cite {MPU95} also obtained new constant scalar curvature metrics by gluing possibly noncompact manifolds of constant positive scalar curvature satisfying a certain nondegeneracy condition.

For equation \eqref{eq1}, we have an analogue of this relation in four dimensions. The roles of scalar curvature and Laplacian operator are replaced by $Q$-curvature and Paneitz operator $P_{g}$, which is a fourth-order operator introduced by Paneitz \cite{Pan08}. The problem of prescribing constant $Q$-curvature is equivalent to solving a fourth order nonlinear PDE. The difficulty comes from the fact that in general, the corresponding Euler functional is not bounded from above or below. Djadli and Malchiodi \cite{DM08} managed to prove the existence of conformal metrics with constant $Q$-curvature under geometric assumptions by min-max theory. The Paneitz operator and its relation with $Q$-curvature in higher dimensions were written down by Branson \cite{Bra87} and Paneitz \cite{Pan08}.

As in four dimensions, we can also ask the same question as above in higher dimensions. Qing and Raske \cite{QR06} answered this question for compact locally conformally flat manifold of dimension $n\geq 5$ with positive scalar curvature, positive constant $Q$-curvature and its Poincar{\'e} exponent less than $\frac{n-4}{2}$.  They also proved the set of these metrics is compact in $C^{\infty}$ topology. In addition, in a special case when Paneitz operator is defined by $\tilde{P}^n_{g}u= \Delta^2_{g}u+\alpha\Delta_{g}u$, where $n\geq 5$ and $\alpha>0$ is a constant, Djadli, Hebey and Ledoux \cite{DHL00} studied the best Sobolev constants and investigated the optimal inequalities. Their result can be applied to obtain a solution of a fourth-order partial differential equation of critical growth. Aside from these somewhat specialized cases, at this time there are no general existence results for constant $Q$-curvature metrics.  

In general, the difficulty of solving higher order differential equations is the lack of a maximum principle. In this paper, we use a gluing method to construct many examples of constant $Q$-curvature metrics for more general compact manifolds of dimension $n\geq 6$ under a nondegeneracy assumption. Our main results are
\begin{theorem}\label{Thm1}
Suppose the dimension $n\geq 6$. Given a compact manifold $(M, g_2)$ of dimension $n$ with constant $Q$-curvature $\nu$. Let $(N, g_1)$ be a compact manifold of dimension $n$. Assume \\
(i) $(M, g_2)$ satisfies the nondegeneracy condition below. \\
(ii) The Paneitz operator $P^{n}_{g_1}$ has positive Green's function $G(x,y)$ on $N$.\\
 Then the connected sum $\widetilde{M}=M\#N$ admits a smooth metric $\tilde{g}$ with constant $Q$- curvature $\nu$.
\end{theorem}
\noindent
\textbf {Nondegeneracy condition}\\
Given a compact manifold $(M,g)$ of dimension $n$, the linearized operator
\begin{equation}\label{lo}
L_{g}:=P^{n}_{g}-\frac{n+4}{2}Q^{n}_{g} : W^{4,2}(M) \rightarrow L^2(M)
\end{equation}
is invertible.\\

In addition to our main theorem, we state the following three propositions which provide us many examples for our gluing problem. We will provide proofs of propositions and examples in Section \ref{ex1} and \ref{ex2}.
\begin{prop}\label{p1}
Given a compact manifold $(M ,g_2)$ of dimension $n\geq 6$ with positive scalar curvature $R_{g_2}>0$ and negative constant $Q$-curvature $Q^{n}_{g_2}<0$, then the assumption (i) in Theorem~\ref{Thm1} is satisfied.
\end{prop}
\begin{prop}\label{p2}
Let $(N,g_1)$ be a compact positive Einstein manifold of dimension $n\geq 6$. Then the Paneitz operator $P^{n}_{g_1}$ has positive Green's function.
\end{prop}
\begin{prop}\label{p3}
Let $(N,g_1)$ be a compact manifold satisfying the assumption (ii) of Theorem \ref{Thm1}. Then the Green's function of Paneitz operator $P^{n}_{g_1}$ has asymptotic expansion $G(p, x)= |x|^{4-n} + O''''(|x|^{6-n})$ as $|x|\rightarrow 0$ if dimension $n\geq 7$ and $G(p, x) = |x|^{4-n} + O''''(|x|^{6-\epsilon-n})$ for any small $\epsilon >0$ as $|x|\rightarrow 0$ if dimension $n=6$, where $\{x^{i}\}$ is a normal coordinate around the point $p$ on $N$.
\end{prop}
\noindent
Combining Propositions \ref{p1}, \ref{p2}, \ref{p3} and Theorem \ref{Thm1}, we immediately have the following corollary.

\begin{cor}\label{Cor1}
Suppose the dimension $n\geq 6$. Assume $(M, g_2)$ has negative constant $Q$-curvature $\nu$ and positive scalar curvature. Let $(N, g_1)$ be a compact positive Einstein manifold. Then there exists a metric with negative constant $Q$-curvature  $\nu$ on the connected sum $\widetilde{M}=N\#M$.
\end{cor}

\noindent
\textbf {Remark 1.}
We can provide many examples of manifolds satisfying the conditions in Corollary \ref{Cor1}. For example, let P be a compact positive Einstein manifold, N be a compact negative Einstein manifold, and F be a compact flat manifold; then the manifold
\begin{equation}
M_2^{n}=\underset{k-copies}{\underbrace{P\times P\times\cdots\times P}}\times\underset{k-copies}{\underbrace{N\times N\times\cdots\times N}}\times\underset{l-copies}{\underbrace{F\times F\times\cdots\times F}}.
\end{equation}
equipped with the product metric has negative constant $Q$-curvature and positive scalar curvature. Thus,  by Corollary \ref{Cor1}, given any compact positive Einstein manifold $M^{n}_1$, there exists a metric with negative constant $Q$-curvature on the connected sum $\widetilde{M}=M^{n}_1\#M^{n}_2$. Proofs of Proposition \ref{p1}, \ref{p2}, \ref{p3} and more details about the examples will be provided in Section \ref{ex1} and \ref{ex2}.\\

\noindent
\textbf {Remark 2.}
When the dimension $n=5$, we do not have a general theorem to control the order of the lower order term in the expansion of the Green's function of Paneitz operator. There are some technical modification we need to make in order to prove Theorem \ref{Thm1} in this case. More discussion about the details will be left in the end of the paper. \\

The outline of this paper is the following: In Section \ref{pre}, we first review some background materials. In Section \ref{app metric}, we define the approximate metric on the connected sum. We show the invertibility of the linearized operator on each manifold in Section \ref{lineach}. In Section \ref{ex1} and \ref{ex2}, we provide examples satisfying assumptions $(i)$ and $(ii)$ respectively. In Section \ref{main}, we prove the invertibility of the linearized operator on the connected sum. After that, we proceed to prove the main theorem using an implicit function theorem argument.
\section{Gluing Construction}\label{gluing}
We will first review some related concepts. Then we can define the approximate metric and discuss the property of linearized operator on each manifold or on the connected sum. In the end, we provide examples for our gluing problem.
\subsection{\bf Background material.}\label{pre}
When dimension $n=2$, if $\tilde{g}=e^{2\psi}g$ is a conformal metric to $g$, then
\begin{equation}
\Delta_{\tilde{g}}u = e^{-2\psi}\Delta_{g}u
\end{equation}
and
\begin{equation}
-\Delta_{g}\psi + \frac{1}{2}R_{g}\psi = \frac{1}{2}R_{\tilde{g}}e^{2\psi},
\end{equation}
where $R_{g}$ is the scalar curvature of metric $g$.

When $n=4$, the Paneitz operator $P^4_{g}$ introduced by Paneitz \cite{Pan08} is a fourth-order differential operator defined by
\begin{equation}
P^4_{g}\phi = \Delta^2_{g}\phi - div_{g}(\frac{2}{3}R_{g}g-2\ric_{g})d\phi,
\end{equation}
where $\ric_{g}$ is the Ricci curvature of $g$. We have
\begin{equation}
P_{\tilde{g}}^4u = e^{-4\psi}P_{g}^4u
\end{equation}
if $\tilde{g}=e^{2\psi}g$. In particular,
\begin{equation}
P_{g}^4\psi + Q^4_{g} = Q_{\tilde{g}}^4e^{4\psi},
\end{equation}
where $Q_{g}^4=\frac{1}{6}(-\Delta_{g}R_{g} + R_{g}^2- 3|\ric_{g}|^2)$ is the $Q$-curvature.
The Paneitz operator and its relation with $Q$-curvature in higher dimensions were written down by Branson \cite{Bra87} and Paneitz \cite{Pan08}. Let (M,g) be a compact manifold of dimension $n\geq 5$, $P_{g}^{n}$ is defined by
\begin{equation}\label{eqP}
P^{n}_{g}u= (-\Delta)^2u + div_{g}(4A_{g} - (n-2)\sigma_1(g^{-1}A_{g})g)du + \frac{n-4}{2}Q_{g}^{n}u,
\end{equation}
where $A_g = \frac{1}{n-2}(\ric_{g}- \frac{1}{2(n-1)}R_{g}g)$ is the Schouten tensor and $\sigma_{i}(g^{-1}A_{g})$ denote the i-th elementary symmetric function of the eigenvalues of $(1,1)$-tensor $g^{-1}A_{g}$ and
\begin{equation}\label{dfQ}
Q^{n}_{g}= -\Delta_{g}\sigma_1(g^{-1}A_{g}) + \frac{n-4}{2}\sigma_1^2(g^{-1}A_{g}) + 4\sigma_2(g^{-1}A_{g})
\end{equation}
is the $Q$-curvature defined for $n\geq 5$. We then have
\begin{equation}\label{eq2}
P_{\tilde{g}}^{n}u = \psi^{-\frac{n+4}{n-4}}P^{n}_{g}(u\psi)
\end{equation}
and
\begin{equation}\label{eqQ}
P_{g}^{n}\psi = \frac{n-4}{2}Q^{n}_{\tilde{g}}\psi^{\frac{n+4}{n-4}}
\end{equation}
if $\tilde{g}=\psi^{\frac{4}{n-4}}g$ is a conformal metric to {g}, for some smooth, positive function $\psi$.
Equations (\ref{eq2}) and (\ref{eqQ}) are higher order analogue of equations (\ref{eqC}) and (\ref{eq1}) respectively, where the roles of $P^{n}_{g}$ and $Q^{n}_{g}$ are replaced by conformal Laplacian $\mathcal{L}_{g}$ and scalar curvature $R_{g}$ respectively.

\subsection{\bf Approximate metric on the connected sum.}\label{app metric}
To construct the approximate metric, first we obtain an asymptotically flat manifold $(N, g_{N})$ out of $(N, g_1)$. Then we would like to scale down $(N, g_{N})$ and glue it onto $(M, g_2)$.

Given a compact manifold $(N, g_1)$ of dimension $n\geq 6$ satisfying the assumption (ii) in Theorem \ref{Thm1}. By Proposition \ref{p3}, we can assume that the Green's function $G(x,y)$ has expansion $G_{p}(x):=G(p,x) = |x|^{4-n} + O''''(|x|^{5-n})$ as $|x|\rightarrow 0$, where $\{x^{i}\}$ is a normal coordinate system around the point $p$ on $N$. With such normal coordinate $\{x^{i}\}$, we have $(g_1)_{ij} = \delta_{ij}- \frac{1}{3}R_{ikjl}(p)x^{k}x^{l}+ O'''(|x|^3)_{ij}$ as $|x| \rightarrow 0$. Let $z^{i}=\frac{x^{i}}{|x|^2}$ denote the inverted normal coordinates near $p$ and let $\Phi(z) = \frac{z}{|z|^2}= x$ denote the inversion. Define $g_{N}=G^{\frac{4}{n-4}}_{p}g_1$. Then we can write down the metric $g_{N}$ with respect to above coordinates in the complement of a large ball.
\begin{align}\label{ASF}
g_{N}&=\Phi^{*}(G^{\frac{4}{n-4}}_{p}g_1)\\
&\notag
= (G_{p}\circ\Phi)^{\frac{4}{n-4}}\Phi^{*}(\{\delta_{ij}- \frac{1}{3}R_{ikjl}(p)x^{k}x^{l}+ O'''(|x|^3)_{ij}\}dx^{i}dx^{j}) \\
&\notag
= (|z|^{n-4}+ O(|z|^{n-5}))^{\frac{4}{n-4}}\{\delta_{ij}- \frac{1}{3}R_{ikjl}(p)\frac{z^{k}z^{l}}{|z|^4}+ O'''(|z|^{-3})_{ij}\}\\
&\notag
*\frac{1}{|z|^2}(\delta_{ip}-\frac{2}{|z|^2}z^{i}z^{p})dz^{p}\frac{1}{|z|^2}(\delta_{jq}-\frac{2}{|z|^2}z^{j}z^{q})dz^{q}\\
&\notag
=|z|^4(1+ O(|z|^{-1}))^{\frac{4}{n-4}}\{\delta_{ij}- \frac{1}{3}R_{ikjl}(p)\frac{z^{k}z^{l}}{|z|^4}+ O'''(|z|^{-3})_{ij}\}\\
&\notag
*\frac{1}{|z|^4}(\delta_{ip}-\frac{2}{|z|^2}z^{i}z^{p})(\delta_{jq}-\frac{2}{|z|^2}z^{j}z^{q})dz^{p}dz^{q}.
\end{align}
Simplify the expansion in equation (\ref{ASF}), we have
\begin{align}
(g_{N})_{pq}(z) &= \delta_{pq} +\frac{C}{|z|}\delta_{pq}-\frac{1}{3}R_{pkql}(p)\frac{z^{k}z^{l}}{|z|^4}+\frac{4}{3}R_{pkjl}(p)\frac{z^{k}z^{l}z^{j}z^{q}}{|z|^6}-\frac{4}{3}R_{ikjl}(p)\frac{z^{k}z^{l}}{|z|^8}z^{i}z^{j}z^{p}z^{q} \\
&\notag
+ O(|z|^{-3})
\end{align}
as $|z|\rightarrow\infty$. This expansion implies that $g_{N}$ is asymptotically flat. Moreover, by equation \eqref{eqQ}, $Q^{n}_{g_{N}} \equiv0$ on $N\setminus\{p\}$.

Let $(M, g_2)$ be a compact manifold with constant $Q$-curvature $\nu$ satisfying the nondegeneracy condition in Theorem \ref{Thm1}. We can choose a normal coordinate system $\{u^{i}\}$ near a point $q$ such that $(g_2)_{ij} = \delta_{ij}- \frac{1}{3}R_{ikjl}(q)u^{k}u^{l}+ O'''(|u|^3)_{ij}$ as $|u| \rightarrow 0$.

Let $a$, $b > 0$ be two small parameters. Combining information from above paragraphs, we have
\begin{equation}
g_{N}= |dz|^2 + \eta_1(z) \text{ for } |z|\geq 1,
\end{equation}
where $$\eta_1(z) = (\eta_1)_{ij}(z) dz^{i}dz^{j}, (\eta_1)_{ij}(z) = O(|z|^{-1})$$
and
\begin{equation}
g_2 = |du|^2 + \eta_2(u)  \text{ for } |u|\leq1,
\end{equation}
where $$\eta_2(u)=(\eta_2)_{ij}(u) du^{i}du^{j}, (\eta_2)_{ij}(u) = O(|u|^2).$$
Topologically, the connected sum $\widetilde{M}=N\# M$ is the result of removing two small balls $B^{n}$ from each manifold centered at $p$ and $q$ respectively, and joining them along their boundaries $S^{n-1}$. We identify an annular region $\{a^{-1}\leq|z|\leq4a^{-1}\}$ in $N$ with a similar annular region $\{b\leq|u|\leq4b\}$ in $M$  by letting $u=abz$. In this annular region, we then have
\begin{align}
a^2b^2g_{N}&=a^2b^2|dz|^2 + a^2b^2(\eta_1)_{ij}(z) dz^{i}dz^{j}\\
&\notag
= |du|^2 + \tilde{\eta_1}_{ij}(u) du^{i}du^{j},
\end{align}
where $\tilde{\eta_1}_{ij}(u) = (\eta_1)_{ij}(a^{-1}b^{-1}u)$.

Geometrically, we would like to glue $a^2b^2g_{N}$ to $g_2$. To do this, pick a cut-off function $\theta_1(t)$, where $0\leq\theta_1\leq1$, with $\theta_1 = 1$ for $t\leq1$ and $\theta_1 = 0$ for $t\geq4$. Set $\theta_2(t) = 1 -\theta_1(t)$. Now, we define the approximate metric $g_{a,b}$ on the connected sum $\widetilde{M}=N\# M$ as the following:
\begin{equation}
g_{a,b}= \begin{cases}g_2 & |u|\geq4b \\ |du|^2 + [\theta_1(\frac{|u|}{b})\tilde{\eta_1}(u) + \theta_2(\frac{|u|}{b})\eta_2(u)] & b\leq|u|\leq4b \\ a^2b^2g_{N} & |z|\leq a^{-1} \end{cases}.
\end{equation}

\subsection{Linearized operator on each weighted space.}\label{lineach}
With the approximate metric $g_{a,b}$ defined in the previous section, we would like to conformally perturb it to a metric $\tilde{g}$ with constant $Q$-curvature $\nu$. That is, to find a conformal metric $\tilde{g}=(1+\phi)^{\frac{4}{n-4}}g_{a,b}$ such that $Q^{n}_{\tilde{g}}= \nu$. If we define the nonlinear map
\begin{equation}
\mathbf{N}_{g}: C^{4,\alpha}(M) \rightarrow C^{0,\alpha}(M)
\end{equation}
by
\begin{equation}\label{eqN1}
\mathbf{N}_{g}[u] := u^{-\frac{n+4}{n-4}}P^{n}_{g}[u]-\frac{n-4}{2}\nu,
\end{equation}
we then convert the geometric problem into solving the following equation:
\begin{equation}\label{eqN}
\mathbf{N}_{g_{a,b}}[1+\phi]= (1+\phi)^{-\frac{n+4}{n-4}}P^{n}_{g_{a,b}}[1+\phi]-\frac{n-4}{2}\nu=0,
\end{equation}
where $\phi$ is small in an appropriate weighted H\"older norm to be defined later. We can also compute its linearized operator.
\begin{equation}
L_{g}[\phi] := \frac{d}{ds}\mathbf N_{g}(1 + s\phi)|_{s=0}=P^{n}_{g}[\phi] - \frac{n+4}{2}Q^{n}_{g}\phi
\end{equation}
\noindent
\textbf {Remark 3.}
The domain of $\mathbf{N}_{g}[u]$ is not the entire space. It is the subset of  $C^{4,\alpha}(M)$ such that $u>0$.\\

The key of solving equation (\ref{eqN}) is proving the linearized operator $L_{g_{a,b}}$ is invertible on the connected sum $\widetilde{M}=N\#M$. To achieve that, we need to show that the linearized operator is invertible on each manifold acting on an appropriate weighted H\"{o}lder space. Denote $N_{p}:=N\setminus\{p\}$ and $M_{q}:= M\setminus\{q\}$. First, we have to define the weight functions $\rho_1$ smoothly on $N_{p}$. Recall that $\{z^{i}\}$ is an inverted normal coordinate near $p$ on $(N_{p}, g_{N})$. Define
\begin{equation}
\rho_1 = |z| \text{ for } |z|\geq2, \rho_1=1 \text{ for }|z|\leq\frac{1}{2}
\end{equation}
and $\rho_1\geq1$ elsewhere. For $\delta\in\mathbb{R}$,
\begin{equation}
||f||_{C^0_{\delta,g_{N}}(N_{p})}\equiv||\rho^{-\delta}_1f||_{C^0(N_{p})}= \sup_{z\in N_{p}}|\rho^{-\delta}_1(z)f(z)|.
\end{equation}
For $0<\alpha<1$, define the semi-norm
\begin{equation}
|f|_{C^{0,\alpha}_{\delta,g_{N}}(N_{p})} \equiv \sup_{z\in N_{p}}\left(\rho^{-\delta+\alpha}_1(z) \sup_{0<4d_{g_{N}}(z,\zeta)\leq\rho_1(z)}\frac{|f(z)-f(\zeta)|}{d_{g_{N}}(z,\zeta)^{\alpha}}\right).
\end{equation}
Finally, we can define the weighted H\"older norm
\begin{equation}
||f||_{C^{k,\alpha}_{\delta,g_{N}}(N_{p})}\equiv \sum_{i=0}^{k}||\nabla^{i}f||_{C^0_{\delta-i,g_{N}}(N_{p})} + |\nabla^{k}f|_{C^{0,\alpha}_{\delta-k,g_{N}}(N_{p})}.
\end{equation}
Furthermore, we would like to define weighted H\"older space $C^{k,\alpha}_{\delta, g_1}(N_{p})$ respect to the background metric $g_1$. We define the weight function $\rho$ on $N_{p}$ as
\begin{equation}
\rho = |x| \text{ for } |x|\leq\frac{1}{2}, \rho=1 \text{ for }|x|\geq2
\end{equation}
and $\rho\geq\frac{1}{2}$ elsewhere, where $\{x^{i}\}$ is a normal coordinate around the point $p$. We then define the weighted H\"older space $C^{k,\alpha}_{\delta, g_1}(N_{p})$ in the same manner as $C^{k,\alpha}_{\delta,g_{N}}(N_{p})$ with the weight function replaced by $\rho$. Note that since $|x|= |z|^{-1}$ near the point $p$, we know $\rho= \rho_1^{-1}$.

Now we can view the nonlinear map in (\ref{eqN}) as a map from
\begin{equation}
\mathbf{N}_{g}[1+\phi]: C^{4,\alpha}_{\delta}\rightarrow C^{0,\alpha}_{\delta-4}
\end{equation}
We then have the following invertibility results.

\begin{prop}\label{p4}
Suppose the compact manifold $(N, g_1)$ satisfies the assumption (ii) of Theorem \ref{Thm1}. Then $L_{g_{N}}: C^{4, \alpha}_{\delta,g_{N}}(N_{p}) \longrightarrow C^{0, \alpha}_{\delta - 4,g_{N}}(N_{p}) $ is invertible if $4-n<\delta<0$ and $n\geq 6$.
\end{prop}
\begin{proof}
We first prove the surjectivity of $L_{g_{N}}$. By equation (\ref{eqQ}), $g_{N}$ has zero $Q$-curvature. We have
\begin{equation}
L_{g_{N}}[u]=P^{n}_{g_{N}}[u]-\frac{n+4}{2}Q^{n}_{g_{N}}u=P^{n}_{g_{N}}[u].
\end{equation}
Thus, we would like to solve
\begin{equation}\label{eqP1}
P^{n}_{g_{N}}[u]=f , \ \forall f\in C^{0, \alpha}_{\delta - 4, g_{N}}(N_{p}).
\end{equation}
Since the weight function for the space $C^{k, \alpha}_{4-\delta, g_1}(N_{p})$ with respect to the metric $g_1$ is $\rho=\rho^{-1}_1$, which is the inverse of the weight function of the space $C^{k, \alpha}_{4-\delta,g_{N}}(N_{p})$, we observe that
\begin{equation}
C^{k, \alpha}_{\delta,g_{N}}(N_{p}) = C^{k, \alpha}_{-\delta, g_1}(N_{p}).
\end{equation}
Recall $g_{N}= G^{\frac{4}{n-4}}_{p}g_1$. By the conformal invariant property of the Panetiz operator, the equation (\ref{eqP1}) is equivalent to
\begin{equation}\label{eqP2}
P^{n}_{g_1}[G_{p}u]=fG_{p}^{\frac{n+4}{n-4}}, \ \forall f\in C^{0, \alpha}_{4-\delta, g_1}(N_{p}).
\end{equation}
With the assumption that $P^{n}_{g_1}$ has positive Green's function $G(x,y)$, by Proposition \ref{p3}, we can assume that $G(x,y)$ has the expansion $G_{p}(x)=|x|^{4-n}+O''''(|x|^{5-n})$  as $|x|\rightarrow 0$. Then, we would like to claim that $fG^{\frac{n+4}{n-4}}_{p}\in C^{0, \alpha}_{-n-\delta, g_1}(N_{p})$. First, we check the $C^{0}_{-n-\delta,g_1}(N_{p})$ norm. After that, we will estimate the H\"older part of the weighted norm. First, we have
\begin{equation}
|fG_{p}^{\frac{n+4}{n-4}}| \leq C(|x|^{4-\delta})(|x|^{-n-4}+O''''(|x|^{-n-3})) = C(|x|^{-\delta-n}+O''''(|x|^{-\delta-n+1}))
\end{equation}
as $|x|\rightarrow 0$. Thus, it's clear that
\begin{equation}
||fG_{p}^{\frac{n+4}{n-4}}||_{C^{0}_{-n-\delta,g_1}(N_{p})}\leq C
\end{equation}
for some constant $C$. Now we proceed to show that
\begin{equation}\label{eqH}
|fG_{p}^{\frac{n+4}{n-4}}|_{C^{0,\alpha}_{-n-\delta,g_1}(N_{p})}\leq C.
\end{equation}
To prove equation (\ref{eqH}), we have to prove
\begin{equation}\label{eqH1}
\sup_{x\in N_{p}}\left(\rho^{n+\delta+\alpha}(x) \sup_{0<4d_{g_1}(x,y)\leq\rho(x)}\frac{|f(x)G_{p}^{\frac{n+4}{n-4}}(x)- f(y)G_{p}^{\frac{n+4}{n-4}}(y)|}{d_{g_1}(x,y)^{\alpha}}\right) \leq C.
\end{equation}
By the fact that Riemannian distance $d_{g_1}(x,y)$ is comparable to the Euclidean metric $|x-y|$ in a small neighborhood of the point $p$ and the condition of $0<4d_{g_1}(x,y)\leq\rho(x)$, we deduce that
\begin{equation}\label{eqb}
\frac{|x|}{|y|}\sim1\ as\ |x|, |y| \rightarrow 0.
\end{equation}
That is, $|x|$ and $|y|$ are comparable as $|x|$, $|y|\rightarrow 0$. Thus, we have
\begin{align}\label{eqH0}
&\rho^{n+\delta+\alpha}(x) \sup_{0<4d_{g_1}(x,y)\leq\rho(x)}\frac{|f(x)G_{p}^{\frac{n+4}{n-4}}(x)- f(y)G_{p}^{\frac{n+4}{n-4}}(y)|}{d_{g_1}(x,y)^{\alpha}}\\
&\notag
\leq \rho^{n+\delta+\alpha}\sup_{0<4d_{g_1}(x,y)\leq\rho(x)}\left(\frac{|f(x)- f(y)||G_{p}^{\frac{n+4}{n-4}}(x)|}{d_{g_1}(x,y)^{\alpha}} +\frac{|f(y)||G_{p}^{\frac{n+4}{n-4}}(x)-G_{p}^{\frac{n+4}{n-4}}(y)|}{d_{g_1}(x,y)^{\alpha}}\right)\\
&\notag
\leq C \left(\rho^{\delta-4+\alpha}(x) \sup_{0<4d_{g_1}(x,y)\leq\rho(x)}\frac{|f(x)- f(y)|}{d_{g_1}(x,y)^{\alpha}}\right)(\rho^{n+4}|G_{p}^{\frac{n+4}{n-4}}(x)| )\\
&\notag
+ C(\rho^{\delta-4}|f(y)|)\left(\rho^{n+4+\alpha}(x) \sup_{0<4d_{g_1}(x,y)\leq\rho(x)}\frac{|G_{p}^{\frac{n+4}{n-4}}(x)-G_{p}^{\frac{n+4}{n-4}}(y)|}{d_{g_1}(x,y)^{\alpha}}\right).
\end{align}
Since $f\in C^{0, \alpha}_{4-\delta, g_1}(N_{p})$ and $G_{p}^{\frac{n+4}{n-4}}(x) = |x|^{-n-4} + O''''(|x|^{-n-3})$ as $|x|\rightarrow 0$, the first term on the right hand side of the last inequality in equation (\ref{eqH0}) is bounded. Thus, from equations (\ref{eqb}) and (\ref{eqH0}), to prove (\ref{eqH1}), it suffices to prove
\begin{equation}\label{eqb1}
\sup_{x\in N_{p}}\left(\rho^{n+4+\alpha}(x) \sup_{0<4d_{g_1}(x,y)\leq\rho(x)}\frac{|G_{p}^{\frac{n+4}{n-4}}(x)- G_{p}^{\frac{n+4}{n-4}}(y)|}{d_{g_1}(x,y)^{\alpha}}\right) \leq C.
\end{equation}
We then perform the following calculations:
\begin{align}\label{esH}
\frac{\big||x|^{-4-n}-|y|^{-4-n}\big|}{d_{g_1}(x,y)^{\alpha}}&\leq C\frac{\big|\frac{1}{|x|}- \frac{1}{|y|} \big|}{|x-y|^{\alpha}}\left(\frac{1}{|x|^{n+3}}+ \frac{1}{|x|^{n+2}|y|}+... +\frac{1}{|y|^{n+3}}\right)\\
&\notag
= C\frac{||x| - |y||}{|x||y||x-y|^{\alpha}}\left(\frac{1}{|x|^{n+3}}+ \frac{1}{|x|^{n+2}|y|}+... +\frac{1}{|y|^{n+3}}\right)\\
&\notag
\leq C\frac{|x - y|^{1-\alpha}}{|x||y|}\left(\frac{1}{|x|^{n+3}}+ \frac{1}{|x|^{n+2}|y|}+... +\frac{1}{|y|^{n+3}}\right)\\
&\notag
\leq C\frac{|x|^{1-\alpha}}{|x||y|}\left(\frac{1}{|x|^{n+3}}+ \frac{1}{|x|^{n+2}|y|}+... +\frac{1}{|y|^{n+3}}\right)\\
&\notag
\leq C\frac{|x|^{-\alpha}}{|y|}\left(\frac{1}{|x|^{n+3}}+ \frac{1}{|x|^{n+2}|y|}+... +\frac{1}{|y|^{n+3}}\right)\\
&\notag
\leq C\frac{|x|^{-\alpha}}{|x|}\left(\frac{1}{|x|^{n+3}}+ \frac{C}{|x|^{n+2}|x|}+... +\frac{C}{|x|^{n+3}}\right)\\
&\notag
\leq C|x|^{-n-4-\alpha}
\end{align}
The second to last inequality holds by equation (\ref{eqb}). With the estimate in equation (\ref{esH}), we basically proved equation (\ref{eqb1}), which implies that
\begin{equation}
fG^{\frac{n+4}{n-4}}_{p}\in C^{0, \alpha}_{-n-\delta, g_1}(N_{p}).
\end{equation}
Therefore, solving the equation (\ref{eqP1}) is equivalent to solving
\begin{equation}\label{eqP3}
P^{n}_{g_1}[\phi] = \tilde{f}, \ \forall \tilde{f}\in C^{0, \alpha}_{-n-\delta, g_1}(N_{p}),
\end{equation}
where $\phi\in C^{4, \alpha}_{-n-\delta + 4, g_1}(N_{p})$. Once we solve equation (\ref{eqP3}), this implies that $u=\frac{\phi}{G_{p}}$ is a solution of equation (\ref{eqP2}).

We have converted the problem into proving the surjectivity of the operator
\begin{equation}
P^{n}_{g_1} : C^{4, \alpha}_{-n-\delta + 4, g_1}(N_{p}) \longrightarrow C^{0, \alpha}_{-n-\delta, g_1}(N_{p}).
\end{equation}
Since $4-n<-n-\delta+4<0$ and there is no indicial root of bilaplacian operator in this range, the linearized operator $P^{n}_{g_1}$ is Fredholm. Note that the Paneitz operator is self-adjoint. Thus, to prove the surjectivity, we need to show the adjoint operator $(P^{n}_{g_1})^{*}=P^{n}_{g_1}:C^{4, \alpha}_{\delta, g_1}(N_{p}) \longrightarrow C^{0, \alpha}_{\delta -4, g_1}(N_{p})$ is injective. Suppose $\eta\in \text{ker} (P_{g_1}^{n})^{*}$. Since $4-n<\delta$ and there is no indicial root of bilaplacian operator in $(4-n, 0)$, $\eta$ has to be in $C^{4, \alpha}_{0, g_1}$, which means $\eta$ has a removable singularity. Since $P^{n}_{g_1}\eta=0$, by standard elliptic regularity theory, we know that $\eta$ is smooth on $N$. We then reduce the problem on the punctured manifold $N\setminus\{p\}$ back to the compact manifold $N$. With the assumption that $P^{n}_{g_1}$ has positive Green's function, we have the invertibility of $P^{n}_{g_1}$ on the compact manifold $N$. Thus, we prove the surjectivity part of this proposition.

The proof of the injectivity is similar to the above argument. Assume $\psi\in \text{ker}P^{n}_{g_{N}}$, where $\psi\in C^{4,\alpha}_{\delta, g_{N}}(N_{p})$, then we have $P^{n}_{g_1}[G_{p}\psi]=0$ by the conformal invariant property mentioned above. It means $G_{p}\psi\in \text{ker}P^{n}_{g_1}$ and $G_{p}\psi\in C^{4,\alpha}_{-n-\delta+4, g_1}(N_{p})$. Since $4-n<\delta<0$, we have $4-n<-n-\delta+4<0$. Again, we know there is no indicial root of bilaplacian operator in $(4-n, 0)$. This implies that $G_{p}\psi\in C^{4, \alpha}_{0, g_1}(N_{p})$, which means $G_{p}\psi$ has a removable singularity. By standard elliptic regularity theory again, we conclude that $G_{p}\psi$ is smooth on $N$. By the invertibility of $P^{n}_{g_1}$ on $N$, we get $G_{p}\psi\equiv 0$, which implies $\psi\equiv 0$.
\end{proof}
We define the weighted H\"older space on $M_{q}$ similarly. First, define weight function $\rho_2$ as
\begin{equation}
\rho_2 = |u| \text{ for }|u|\leq\frac{1}{2}, \rho_2 = 1\text{ for }|u|\geq2
\end{equation}
and $\rho_2\geq\frac{1}{2}$ for $|u|\geq\frac{1}{2}$. For $\delta\in\mathbb{R}$, one can also define
\begin{equation}
||f||_{C^0_{\delta,g_2}(M_{q})}\equiv||\rho^{-\delta}_2f||_{C^0(M_{q})}= \sup_{u\in M_{q}}|\rho^{-\delta}_2(u)f(u)|.
\end{equation}
For $0<\alpha<1$, define the semi-norm
\begin{equation}
|f|_{C^{0,\alpha}_{\delta,g_2}(M_{q})} \equiv \sup_{u\in M_{q}}\left(\rho^{-\delta+\alpha}_2(u) \sup_{0<4d_{g_2}(u,v)\leq\rho_2(u)}\frac{|f(u)-f(v)|}{d_{g_2}(u,v)^{\alpha}}\right).
\end{equation}
Finally, we can define the norm
\begin{equation}
||f||_{C^{k,\alpha}_{\delta,g_2}(M_{q})}\equiv \sum_{i=0}^{k}||\nabla^{i}f||_{C^0_{\delta-i,g_2}(M_{q})} + |\nabla^{k}f|_{C^{0,\alpha}_{\delta-k,g_2}(M_{q})}.
\end{equation}
We also have the invertibility result for the manifold $(M, g_2)$.

\begin{prop}\label{p5}
Suppose the compact manifold $(M, g_2)$ satisfies the nondegeneracy condition in Theorem \ref{Thm1}. Then the linearized operator $L_{g_2}: C^{4, \alpha}_{\delta, g_2}(M_{q})\longrightarrow C^{0,\alpha}_{\delta-4, g_2}(M_{q})$ is invertible if $4-n<\delta<0$ and $n\geq6$.
\end{prop}
\begin{proof}
First, $L_{g_2}[u]=P^{n}_{g_2}[u]-\frac{n+4}{2}Q^{n}_{g_2}u$ is a Fredholm operator since $4-n<\delta<0$. Moreover, it is self-adjoint. To prove the surjectivity of $L_{g_2}$, it is equivalent to proving the injectivity of the adjoint operator
\begin{equation}
L^{*}_{g_2}=L_{g_2} : C^{4, \alpha}_{4-\delta-n, g_2}(M_{q})\longrightarrow C^{0,\alpha}_{-\delta-n, g_2}(M_{q})
\end{equation}
Assume $\phi\in \text{ker}L_{g_2}^{*}$, where $\phi\in C^{4,\alpha}_{4-\delta-n, g_2}(M_{q})$. Since $4-n<4-\delta-n<0$ and there is no indicial root in this range, $\phi$ is smooth across the singularity point $q$ on $M$ by the standard elliptic regularity theory as in the proof of Proposition \ref{p4}. With the assumption that $L_{g_2}$ is invertible on the compact manifold $M$, we thus prove the surjectivity part of this proposition.

To prove the injectivity, observe that if $\eta\in \text{ker}L_{g_2}$, where $\eta\in C^{4, \alpha}_{\delta, g_2}(M_{q})$, $\eta$ is smooth on $M$ because of $4-n<\delta<0$. Then by a similar argument as above, we get $\eta\equiv0$.
\end{proof}

\subsection{\bf Manifolds satisfying the assumption (i) of Theorem~\ref{Thm1}}\label{ex1}
Before providing examples, let's first recall the statement of Proposition \ref{p1} and provide the proof.

\newtheorem*{p1}{Proposition~\ref{p1}}
\begin{p1}
Given a compact manifold $(M,g_2)$ of dimension $n\geq 6$ with positive scalar curvature $R_{g_2}>0$ and negative constant $Q$-curvature $Q^{n}_{g_2}<0$, then the assumption (i) in Theorem~\ref{Thm1} is satisfied.
\end{p1}
\begin{proof}
With the Bochner formula
\begin{equation}
\int_{M} \ric_{g_2}(\nabla\phi,\nabla\phi) = \int_{M}(\Delta_{g_2}\phi)^2 -|\nabla^2_{g_2}\phi|^2
\end{equation}
and the linearized operator from equation (\ref{lo}), we have
\begin{align}
&\int_{M} \phi L_{g_2}(\phi) dV_{g_2} \\
&\notag
=\int_{M}(\Delta_{g_2}\phi)^2 + a_{n}R_{g_2}|\nabla\phi|^2 - \frac{4}{n-2}\ric_{g_2}(\nabla\phi, \nabla\phi) - 4Q^{n}_{g_2}\phi^2dV_{g_2}\\
&\notag
=\int_{M}(\Delta_{g_2}\phi)^2 + a_{n}R_{g_2}|\nabla\phi|^2 + \frac{4}{n-2}|\nabla^2_{g_2}\phi|^2 -\frac{4}{n-2}(\Delta_{g_2}\phi)^2 - 4Q^{n}_{g_2}\phi^2 dV_{g_2}\\
&\notag
= \int_{M}(1-\frac{4}{n-2})(\Delta_{g_2}\phi)^2 + a_{n}R_{g_2}|\nabla\phi|^2 + \frac{4}{n-2}|\nabla^2_{g_2}\phi|^2- 4Q^{n}_{g_2}\phi^2dV_{g_2}\\
&\notag
\geq0,
\end{align}
where $a_{n}= \frac{(n-2)^2 + 4}{2(n-1)(n-2)}> 0$. The last inequality holds by the fact that $R_{g_2}>0$, $Q^{n}_{g_2} <0$ and $n\geq6$.
This shows that $L_{g_2}$ is strictly positive which implies $0\not\in \text{spec}(L_{g_2})$. By the facts that $L_{g_2}$ is self-adjoint and $M$ is compact, $L_{g_2}$ is invertible.
\end{proof}
Followed from Proposition \ref{p1}, we immediately obtain the following example satisfying the nondegeneracy condition in Theorem \ref{Thm1}. \\

\noindent
\textbf{Example 1.}\\
Let P be a compact positive Einstein manifold, N be a compact negative Einstein manifold and F be a compact flat manifold. Given a $n$-dimensional manifold
\begin{equation}
M_2^{n}:=\underset{k-copies}{\underbrace{P\times P\times\cdots\times P}}\times\underset{k-copies}{\underbrace{N\times N\times\cdots\times N}}\times\underset{l-copies}{\underbrace{F\times F\times\cdots\times F}}.
\end{equation}
with product metric $g$, we can rescale $(M_2^{n} , g)$ to have diagonal Schouten tensor $g^{-1}\circ A(g)=(a_{ij})$ where $a_{ii}=\frac{1}{2}$ for $i=1 \ldots k$, $a_{jj}= -\frac{1}{2} +\epsilon$ for $j =(k+1) \ldots2k$ and $a_{mm}= 0$ for $m=(2k+1) \ldots n$. By definition, we know
\begin{align}
R_{g}=2(n-1)\sigma_1(g^{-1}\circ A(g)) &= 2(n-1)tr(g^{-1}\circ A(g)) \\
&\notag
= 2(n-1)\left(\frac{k}{2} + k(-\frac{1}{2}+\epsilon)\right) \\
&\notag
=2(n-1)k\epsilon.
\end{align}
Also, recall the definition for $Q^{n}_{g}$ in equation (\ref{dfQ}). We have
\begin{align}
Q^{n}_{g}&= -\Delta_{g}\sigma_1(g^{-1}A_{g}) + \frac{n-4}{2}\sigma_1^2(g^{-1}A_{g}) + 4\sigma_2(A_{g})\\
&\notag
= \frac{n-4}{2}(k\epsilon)^2 + 4\left[\frac{k(k-1)}{2}\frac{1}{4}+ \frac{k(k-1)}{2}(-\frac{1}{2}+\epsilon)^2 + k^2(\frac{1}{2})(-\frac{1}{2}+\epsilon)\right]\\
&\notag
= -k - 2k\epsilon^2 + 2k\epsilon +\frac{n}{2}k^2\epsilon^2.
\end{align}
Thus, for any given value of $k$, we can choose $\epsilon=\epsilon(k,n)>0$ small enough such that $(M_2^{n},g)$ satisfies $R_{g}>0$ and $Q^{n}_{g}<0$. Thus, by Proposition \ref{p1}, $(M_2^{n},g)$ satisfies the assumption (i) of Theorem~\ref{Thm1}.

\subsection{\bf Manifolds satisfying the assumption (ii) of Theorem~\ref{Thm1}}\label{ex2}
By Proposition ~\ref{p2} 
, we can show that any compact, positive Einstein manifold $(N, g_1)$ of dimension $n\geq 6$ satisfies the assumption (ii) of Theorem \ref{Thm1}. We provide the proof in the following.

\newtheorem*{p2}{Proposition~\ref{p2}}
\begin{p2}
Let $(N,g_1)$ be a compact positive Einstein manifold of dimension $n\geq 6$. Then the Paneitz operator $P^{n}_{g_1}$ has positive Green's function.
\end{p2}
\begin{proof}
Without loss of generality, we may assume $\ric_{g_1}= (n-1)g_1$. From the computations in \cite{DHL00}, we know
\begin{equation}
Q_{g_1}^{n} = \frac{n(n^2-4)}{8}
\end{equation}
and
\begin{align}
P_{g_1}^{n}u &= \Delta_{g_1}^2 u - \frac{n^2-2n-4}{2}\Delta_{g_1} u + \frac{n(n-4)(n^2-4)}{16}u\\
&\notag
= \left[- \Delta_{g_1} + (\frac{1}{4}n^2-\frac{n}{2}-2) \right]\left[- \Delta_{g_1} + (\frac{1}{4}n^2-\frac{n}{2}) \right] u.
\end{align}
By direct calculations, we know that $P^{n}_{g_1}$ is strictly positive. Since $P_{g_1}^{n}$ is self-adjoint and $N$ is compact, we conclude that $P^{n}_{g_1}$ is invertible. Thus, by a similar argument in section 11 in Robert's notes \cite{Rob09}, the Green's function of the Paneitz operator $G(x,y)$ exists on $(N,g_1)$.

Fix an arbitrary point $p$ on $N$. Denote $G_{p}(x) = G(p, x)$. We define
\begin{equation}
F(x) := \left[- \Delta_{g_1} + (\frac{1}{4}n^2-\frac{n}{2})\right]G_{p}(x).
\end{equation}
By definition, we know
\begin{equation}
\left[- \Delta_{g_1} + (\frac{1}{4}n^2-\frac{n}{2}-2)\right]F(x) \equiv 0
\end{equation}
on $N\setminus\{p\}$, where $p$ is the singularity point of Green's function. By maximum principle, we know either $F(x) >0$ or $F(x) <0$ on $N\setminus\{p\}$. We know $G_{p}(x)= |x|^{4-n} + o(|x|^{4-n})$ as $|x|\rightarrow 0$, where $\{x^{i}\}$ is a normal coordinate near the singularity point $p$. It turns out that
\begin{equation}
F(x)= \left[- \Delta_{g_1} + (\frac{1}{4}n^2-\frac{n}{2})\right]G_{p}(x) = 2(n-4)|x|^{2-n}+ o(|x|^{2-n}),
\end{equation}
as $|x|\rightarrow 0$. Thus, $F(x)\rightarrow +\infty$ as $|x|\rightarrow 0$, which means that $F(x)>0$ on $N\setminus\{p\}$. Now by maximum principle again, we know $G_{p}(x)$ attains its minimum at an interior point $x_{m}$. Then we have $(\Delta_{g_1}G_{p})(x_{m}) \geq 0$ which implies $(\frac{1}{4}n^2-\frac{n}{2})G_{p}(x_{m}) >0$. That is, $G_{p}(x_{m}) >0$. Thus, we have $G_{p}(x)>0$ on $N\setminus\{p\}$.
\end{proof}

\newtheorem*{p3}{Proposition~\ref{p3}}
\begin{p3}
Let $(N,g_1)$ be a compact manifold satisfying the assumption (ii) of Theorem \ref{Thm1}. Then the Green's function of Paneitz operator $P^{n}_{g_1}$ has asymptotic expansion $G(p, x)= |x|^{4-n} + O''''(|x|^{6-n})$ as $|x|\rightarrow 0$ if dimension $n\geq 7$ and $G(p, x) = |x|^{4-n} + O''''(|x|^{6-\epsilon-n})$ for any small $\epsilon >0$ as $|x|\rightarrow 0$ if dimension $n=6$, where $\{x^{i}\}$ is a normal coordinate around the point $p$ on $N$.
\end{p3}
\begin{proof}
Recall
\begin{equation}
P^{n}_{g}u= (-\Delta)^2u + div_{g}(4A_{g} - (n-2)\sigma_1(g^{-1}A_{g})g)du + \frac{n-4}{2}Q_{g}^{n}u.
\end{equation}
Since $g_1$ is a smooth metric, we have
\begin{equation}
P_{g_1}^{n}u= \Delta_{g_1}^2 u + a_{ij}\nabla^2_{ij}u + b_{i}\nabla_{i}u+ cu,
\end{equation}
where $a_{ij}$, $b_{i}$ and $c$ are all smooth functions on $(N, g_1)$.

We can choose a normal coordinate system $\{x^{i}\}$ near a point $p$ such that $(g_1)_{ij} = \delta_{ij}- \frac{1}{3}R_{ikjl}(p)x^{k}x^{l}+ O'''(|x|^3)_{ij}= \delta_{ij}+ h_{ij}$ as $|x| \rightarrow 0$. We know
\begin{align}
\Delta_{g_1}^2f &= \Delta^2f + h\ast\nabla^4f + \nabla h\ast\nabla^3f + \nabla^2h\ast\nabla^2f + \nabla^3h\ast\nabla f \\
&\notag
+ h\ast\nabla h\ast\nabla^3f + \nabla h\ast\nabla h\ast\nabla^2f + \nabla h\ast\nabla^2h\ast\nabla f + l.o.t.
\end{align}
where $\Delta$ and $\nabla$ are the Laplacian and connection with respect to the Euclidean metric respectively.
Thus we have
\begin{align}
\Delta_{g_1}^2(|x|^{4-n}) &= \Delta^2(|x|^{4-n}) + O(|x|^2)\nabla^4(|x|^{4-n}) + O(|x|)\nabla^3(|x|^{4-n})\\
&\notag
+ O(1)\nabla^2(|x|^{4-n}) + O(1)\nabla(|x|^{4-n}) + O(|x|^3)\nabla^3(|x|^{4-n})\\
&\notag
+ O(|x|^2)\nabla^2(|x|^{4-n}) + O(|x|)\nabla(|x|^{4-n}) + l.o.t. \\
&\notag
= C_1(|x|^{2-n}) + O(|x|^{3-n}).
\end{align}
as $|x|\rightarrow 0$. Thus,
\begin{equation}
P^{n}_{g_1} (|x|^{4-n}) = C_1|x|^{2-n} + O(|x|^{3-n})
\end{equation}
as $|x|\rightarrow 0$. By a similar argument in the proof of Proposition \ref{p4} and a similar calculation as equation (\ref{esH}), we can show that there exists a function $\eta\in C^{0, \alpha}_{2-n, g_1}(N_{p})$ such that $P^{n}_{g_1}(|x|^{4-n}) = \eta$, as $|x|\rightarrow 0$. If $n\geq 7$, $4-n< 6-n < 0$. By the fact that there is no indicial roots in the interval $(4-n, 0)$, we know $P^{n}_{g_1}: C^{4,\alpha}_{6-n, g_1}(N_{p})\rightarrow C^{0,\alpha}_{2-n, g_1}(N_{p})$ is invertible. Thus, there exists a unique function $\phi\in C^{4,\alpha}_{6-n, g_1}(N_{p})$ such that $P^{n}_{g_1}(|x|^{4-n} - \phi)= 0$. It is equivalent to say that to get $P^{n}_{g_1}(G_{p}) = 0$, we should have the expansion $G_{p}(x) = |x|^{4-n} + O''''(|x|^{6-n})$ as $|x|\rightarrow 0$.

If $n=6$, we know $P^{n}_{g_1}: C^{4,\alpha}_{6-\epsilon-n, g_1}(N_{p})\rightarrow C^{0,\alpha}_{2-\epsilon-n, g_1}(N_{p})$ is invertible for any small $\epsilon >0$. Thus, there exists a unique function $\psi\in C^{4,\alpha}_{6-\epsilon-n, g_1}(N_{p})$ such that $P^{n}_{g_1}(|x|^{4-n} - \psi) = 0$. It is equivalent to say that if $n=6$, we have the expansion $G_{p}(x) = |x|^{4-n} + O''''(|x|^{6-\epsilon-n})$ for any small $\epsilon>0$ as $|x|\rightarrow 0$.
\end{proof}


\section{Proof of the main theorem }\label{main}
The core of solving the equation (\ref{eqN}) is to show the linearized operator $L_{g_{a,b}}$ is invertible on the connected sum manifold $\widetilde{M}:=N\#M$. The strategy is to construct an approximate inverse of $L_{g_{a,b}}$ by gluing the inverse operators of $L_{a^2b^2g_{N}}$ and $L_{g_2}$ we obtained from previous sections. Then, we can get the controlled inverse operator of $L_{g_{a,b}}$. In this section, we will first define the weighted H\"older space on $\widetilde{M}$, then we proceed to prove the invertibility of the linearized operator $L_{g_{a,b}}$. Later in Section \ref{IFT}, we show the existence of the solution to equation (\ref{eqN}) using an implicit function theorem argument.

\subsection{Weighted spaces on $\widetilde{M}$}
The weighted spaces that we use on the connected sum $\widetilde{M}:=N\#M$ are glued versions of the weighted spaces defined on  $N_{p}$ and on $M_{q}$. Recall that we have chosen a normal coordinates $\{u^{i}\}$ around $q$ on $M$, defined for $|u| < 1$. To obtain the connected sum manifold, we are gluing a scaled down version of $(N, g_{N})$ to $(M, g_2)$. In terms of the inverted normal coordinate $\{z^{i}\}$ on $N_{p}$, we do the gluing by identifying the annuli $\{b < |u| < 4b\}$ in $M$ and $\{a^{-1} < |z| < 4a^{-1}\}$ in $N$ using the map
$$u=abz.$$
Let
\begin{equation}
w=w^{a,b} = \begin{cases} ab\rho_1(z) & |z|\leq2a^{-1} \\ \rho_2(u) & |u|\geq2b \end{cases}
\end{equation}
be the weight function on $\widetilde{M}$. Note that
$$ ab\leq w\leq 1.$$
Then we define the weighted space on the connected sum $\widetilde{M}=N\#M$ as the following.
For $\delta\in\mathbb{R}$,
\begin{equation}
 ||f||_{C^0_{\delta}(\widetilde{M})}\equiv||w^{-\delta}f||_{C^0(\widetilde{M})}= \sup_{x\in \widetilde{M}}|w^{-\delta}(x)f(x)|.
\end{equation}
For $0<\alpha<1$, define the semi-norm
\begin{equation}
|f|_{C^{0,\alpha}_{\delta}(\widetilde{M})} \equiv \sup_{x\in\widetilde{M}}\left(w^{-\delta+\alpha}(x) \sup_{0<4d_{g_{a,b}}(x,y)\leq w(x)}\frac{|f(x)-f(y)|}{d_{g_{a,b}}(x,y)^{\alpha}}\right).
\end{equation}
Finally, we can define the weighted H\"older norm on $\widetilde{M}$ as
\begin{equation}
||f||_{C^{k,\alpha}_{\delta}(\widetilde{M})}\equiv \sum_{i=0}^{k}||\nabla^{i}f||_{C^0_{\delta-i}(\widetilde{M})} + |\nabla^{k}f|_{C^{0,\alpha}_{\delta-k}(\widetilde{M})}
\end{equation}

\subsection{Analysis of linearized operator.}
The key of proving Theorem \ref{Thm1} is the following theorem which shows the invertibility of the linearized operator $L_{g_{a,b}}$.

\begin{theorem}\label{Thm2}
Suppose $4-n<\delta<0$. There exists a constant $C>0$, independent of $a$ and $b$ such that, if $a=b^4$ and $b$ is small enough, then the linearized operator $L_{g_{a,b}}: C^{4,\alpha}_{\delta}(\widetilde{M}) \longrightarrow C^{0,\alpha}_{\delta-4}(\widetilde{M})$ is invertible and we have a uniform bound $||L_{g_{a,b}}^{-1}||_{C^{4,\alpha}_{\delta}(\widetilde{M})}\leq C$ for its inverse.
\end{theorem}
\begin{proof}
The proof follows the argument from \cite{DK90} for gluing solutions of linear equations by constructing an approximate inverse operator $H_0$ for the linearized operator $L_{g_{a,b}}$.

Recall $\theta_1$ and $\theta_2$ are the cut-off functions defined in Section \ref{app metric}. Denote
\begin{equation}
\gamma_1(z) := \theta_1\left(\frac{|u|}{b}\right), \gamma_2(u) := \theta_2\left(\frac{|u|}{b}\right)
\end{equation}
We also need to define the new cut-off functions $\beta_1$ and $\beta_2$ as
\begin{equation}
\beta_1(z)=\theta_1\left(\left(\frac{|u|}{4b}\right)^{\lambda}\right), \beta_2(u) = \theta_2\left(4\left(\frac{|u|}{b}\right)^{\lambda}\right).
\end{equation}
where $\lambda>0$ is a small constant to be chosen. In fact, we can choose $\lambda = O(\frac{1}{\log b})$ such that $\beta_1=1$ on the support of $\gamma_1$, $\nabla\beta_1$ support in a region slightly smaller than $B_{\sqrt{b}}\setminus B_{4b}$ and $\beta_1=0$ outside $B_{\sqrt{b}}$. Similarly, with the chosen $\lambda$, we can make $\beta_2=1$ on the support of $\gamma_2$, $\nabla\beta_2$ support in a region slightly smaller than $B_{b}\setminus B_{b^2}$ and $\beta_2=0$ in $B_{b^2}$. Thus, we have $\gamma_1 + \gamma_2 = 1$, $\beta_1\gamma_1 = \gamma_1$ and $\beta_2\gamma_2=\gamma_2$.

With all the cut-off functions defined above, one can construct an approximate inverse operator $H_0$ for the linearized operator $L_{g_{a,b}}$. By Proposition \ref{p4} and Proposition \ref{p5}, one can define
\begin{equation}
H_0 :=\beta_1(a^4b^4H_1)\gamma_1 + \beta_2H_2\gamma_2,
\end{equation}
where $H_1$, $H_2$ are the inverse operators of $L_{g_{N}}$ and $L_{g_2}$ respectively. In this case, $a^4b^4H_1$ would be the inverse operator of $L_{a^2b^2g_{N}}$.
We then have
\begin{align}\label{eqL}
L_{g_{a,b}}(H_0) &= L_{g_{a,b}}(a^4b^4\beta_1H_1\gamma_1) + L_{g_{a,b}}(\beta_2H_2\gamma_2)\\
&\notag
=\beta_1L_{g_{a,b}}(a^4b^4H_1\gamma_1) + \beta_2L_{g_{a,b}}(H_2\gamma_2) + [L_{g_{a,b}},\beta_1]a^4b^4H_1\gamma_1\\
&\notag
+[L_{g_{a,b}},\beta_2]H_2\gamma_2.
\end{align}
First, we need the following lemma.

\begin{lemma}\label{L1}
Suppose $4-n<\delta<0$, $a=b^4$ and $b$ is small enough. Then
\begin{equation}
||L_{g_{a,b}}f - L_{g_2}f||_{C^{0,\alpha}_{\delta-4}(\widetilde{M})} = o(1)||f||_{C^{4,\alpha}_{\delta}(\widetilde{M})}
\end{equation}
if $f$ lives in the support of $\beta_2$ and
\begin{equation}
||L_{g_{a,b}}f - L_{a^2b^2g_{N}}f||_{C^{0,\alpha}_{\delta-4}(\widetilde{M})} = o(1)||f||_{C^{4,\alpha}_{\delta}(\widetilde{M})}
\end{equation}
if $f$ lives in the support of $\beta_1$, where $o(1)$ means it is a constant which goes to zero as $b\rightarrow 0$.
\end{lemma}
\begin{proof}
Recall that the approximate metric $g_{a,b}$ is defined by
\begin{equation}
g_{a,b}= \begin{cases}g_2 & |u|\geq4b \\ |du|^2 + [\theta_1(\frac{|u|}{b})\tilde{\eta_1}(u) + \theta_2(\frac{|u|}{b})\eta_2(u)] & b\leq|u|\leq4b \\ a^2b^2g_{N} & |z|\leq a^{-1} \end{cases}.
\end{equation}
In the annular region $A=\{b\leq|u|\leq4b\}$, $g_{a,b}$ looks like a metric in the form of $g+h$, where $h$ is a (0, 2)-tensor. So, we would like to calculate the curvatures of $g+h$ in terms of  curvatures of $g$ and derivatives of $h$. Before proving the lemma, let\rq{}s review some of the basic calculations.

For Christoffel symbols, we have
\begin{equation}
\Gamma(g+h)^{k}_{ij} = \Gamma(g)^{k}_{ij} + \frac{1}{2}(g+h)^{km}\{\nabla_{j}h_{im}+ \nabla_{i}h_{jm}- \nabla_{m}h_{ij}\}.
\end{equation}
Thus,
\begin{equation}
\nabla_{g+h}T = \nabla_{g}T + (g+h)^{-1}\ast\nabla_{g}h\ast T,
\end{equation}
where $T$ is any tensor field.
By the formula for $(1,3)$-curvature tensor in terms of the Christoffel symbols, we have
\begin{equation}
Rm_{g+h} = Rm_{g} + (g+h)^{-1}\ast\nabla^2h + (g+h)^{-2}\ast\nabla h\ast\nabla h,
\end{equation}
\begin{equation}
\ric_{g+h} = \ric_{g} + \delta_{ik}(g+h)^{-1}\ast\nabla^2h + \delta_{ik}(g+h)^{-2}\ast\nabla h\ast\nabla h,
\end{equation}
and
\begin{equation}
R_{g+h} = R_{g} + (g^{-2}\ast h +\sum_{k\geq2}g^{-k-1}\ast h^{k})\ast\ric_{g} + \delta_{ik}(g+h)^{-2}\ast\nabla^2h +\delta_{ik} (g+h)^{-3}\ast\nabla h\ast\nabla h.
\end{equation}
Thus,
\begin{align}
\Delta_{g+h}R_{g+h} &= (g+h)^{-1}\ast \nabla^2_{g+h}R_{g+h} \\
&\notag
=(g+h)^{-1}(\partial^2R_{g+h} -\Gamma_{g+h}\partial R_{g+h})  \\
&\notag
=(g+h)^{-1}\partial^2R_{g+h}- (g+h)^{-1}(\Gamma(g)^{k}_{ij} \\
&\notag
+ \frac{1}{2}(g+h)^{km}\{\nabla_{j}h_{im}+ \nabla_{i}h_{jm}- \nabla_{m}h_{ij}\})\partial R_{g+h}.
\end{align}
In $A = \{b\leq|u|\leq 4b\}$, $g_{a,b}= \delta_{ij}+\theta_1(\frac{|u|}{b})\tilde{\eta_1}_{ij}(u) + \theta_2(\frac{|u|}{b})\eta_2(u)= \delta_{ij}+ \phi(u)$. If we choose $a=b^4$ and $b$ is small enough, then we have the following estimates for the derivatives of $\phi$.
\begin{align}
|\phi'|&=|b^{-1}\theta_1'\tilde{\eta_1} + \theta\tilde{\eta_1}' + b^{-1}\theta_2'\eta_2+ \theta_2\eta_2'|\\
&\notag
=O(b^{-1})O((ab|u|^{-1}) + O(ab|u|^{-2}) + O(b^{-1}|u|^2) + O(|u|) \\ 
&\notag
=O(ab^{-1}) + O(b)\\
&\notag
=O(b),
\end{align}
\begin{align}
|\phi''| &= |b^{-2}\theta_1''\tilde{\eta_1} + 2b^{-1}\theta_1'\tilde{\eta_1}' + \theta_1\tilde{\eta_1}'' + b^{-2}\theta_2''\eta_2 + 2b^{-1}\theta_2'\eta_2' + \theta_2\eta_2''|\\
&\notag
=O(b^{-2})O((ab|u|^{-1}) + O(b^{-1})O(ab|u|^{-2}) + (ab|u|^{-3}) + O(b^{-2}|u|^2) \\
&\notag
+ O(b^{-1})O(|u|) + O(1) \\
&\notag
=O(ab^{-2}) + O(1)\\
&\notag
=O(1),
\end{align}
\begin{equation}
|\phi'''| = O(ab^{-3})+O(b^{-1})=O(b^{-1}),
\end{equation}
and
\begin{equation}
|\phi''''| = O(ab^{-4})+O(b^{-2})=O(b^{-2}).
\end{equation}
With the above estimates, we can compute the curvatures of metric $g_{a,b}$ in the annular region $A$.
\begin{equation}
|R_{g_{a,b}}| = |\delta_{ik}(g_{a,b})^{-2}\ast\nabla^2\phi +\delta_{ik} (g_{a,b})^{-3}\ast\nabla\phi\ast\nabla\phi|= O(1),
\end{equation}
\begin{align}
|\partial R_{g_{a,b}}| &=| -2\delta_{ik}(g_{a,b})^{-3}\ast\nabla^2\phi\ast\nabla\phi +\delta_{ik}(g_{a,b})^{-2}\ast\nabla^3\phi \\
&\notag
+ 2\delta_{ik} (g_{a,b})^{-3}\ast\nabla^2\phi\ast\nabla\phi -3\delta_{ik}(g_{a,b})^{-4}\ast\nabla\phi\ast\nabla\phi\ast\nabla\phi| \\
&\notag
=O(b^{-1}),
\end{align}
and
\begin{equation}
|\Delta_{g_{a,b}}R_{g_{a,b}}|=O(b^{-2}).
\end{equation}
Recall that
\begin{align}
Q^{n}_{g}&= -\Delta_{g}\sigma_1(g^{-1}A_{g}) + \frac{n-4}{2}\sigma_1^2(g^{-1}A_{g}) + 4\sigma_2(A_{g})\\
&\notag
=-\frac{1}{2(n-1)}\Delta_{g}R_{g} + a_{n}R^2_{g} - \frac{2}{(n-2)^2}|\ric_{g}|^2,
\end{align}
for some dimensional constant $a_{n}$. Thus, one can compute the $Q$-curvature for different regions on the connected sum $\widetilde{M}$.
\begin{equation}
|Q^{n}_{g_{a,b}}| = \begin{cases} \nu & |u|\geq4b \\ O(b^{-2}) & A=\{b\leq|u|\leq4b\} \\ 0 & |z|\leq a^{-1} \end{cases},
\end{equation}
\begin{equation}
|Q^{n}_{g_{a,b}}-\nu| = \begin{cases} 0 & |u|\geq4b \\ O(b^{-2}) & A=\{b\leq|u|\leq4b\} \\ \nu &|z|\leq a^{-1} \end{cases},
\end{equation}
and
\begin{equation}
||Q^{n}_{g_{a,b}}-\nu||_{C^{0,\alpha}_{\delta-4}(\widetilde{M})} = O(b^{2-\delta}).
\end{equation}
We can proceed to compare $L_{g_{a,b}}$ and $L_{g_2}$.
In the support of $\beta_2$, $g_{a,b}=g_2+ h$, where $|h| =|\theta_1(\frac{|u|}{b})(\tilde{\eta_1}(u) - \eta_2(u))|=O(b^2)$. Thus, $|\nabla_{g_2}h| = O(b)$, $|\nabla^2_{g_2}h|=O(1)$ and $|\nabla^3_{g_2}h|=O(b^{-1})$.  Recall the formula of the Paneitz operator from equation (\ref{eqP}). We can also express the Paneitz operator in terms of the scalar curvature $R_{g}$ and Ricci curvature $\ric_{g}$. That is,
\begin{equation}
P_{g}^{n}u = (-\Delta_{g})^2 u - div_{g}\left(b_{n}R_{g}g - \frac{4}{(n-2)} \ric_{g}\right) du + \frac{(n-4)}{2}Q_{g}^{n}u,
\end{equation}
for some dimensional constant $b_{n}$. Thus, we have
\begin{align}\label{eqc1}
&||L_{g_{a,b}}f - L_{g_2}f||_{C^{0,\alpha}_{\delta-4}(\widetilde{M})}\\
&\notag
\leq ||\Delta^2_{g_{a,b}}f - \Delta^2_{g_2}f|| +C||g_{a,b}^{-1}\nabla_{g_{a,b}}R_{g_{a,b}}\nabla_{g_{a,b}}f - g_2^{-1}\nabla_{g_2}R_{g_2}\nabla_{g_2}f||\\
&\notag
+C||R_{g_{a,b}}\Delta_{g_{a,b}}f-R_{g_2}\Delta_{g_2}f||+C ||g_{a,b}^{-1}\nabla_{g_{a,b}}\ric_{g_{a,b}}\nabla_{g_{a,b}}f - g_2^{-1}\nabla_{g_2}\ric_{g_2}\nabla_{g_2}f ||\\
&\notag
+C||g_{a,b}^{-1}g_{a,b}^{-1}\ric_{g_{a,b}}\nabla_{g_{a,b}}^2f-g_2^{-1}g_2^{-1}\ric_{g_2}\nabla_{g_2}^2f || + C||(Q^{n}_{g_{a,b}}-Q^{n}_{g_2})f||.
\end{align}
Let's estimate term by term. First, we have the following calculations:
\begin{align}
\Delta_{g+h}f&= (g+h)^{-1}\nabla_{g+h}\nabla f \\
&\notag
=  (g+h)^{-1}[\nabla^2f + (g+h)^{-1}\ast\nabla_{g}h\ast\nabla f] \\
&\notag
= (g^{-1}+ g^{-2}\ast h +\sum_{k\geq2}g^{-k-1}\ast h^{k})[\nabla^2f + (g+h)^{-1}\ast\nabla_{g}h\ast\nabla f] \\
&\notag
= \Delta_{g}f + g^{-1}(g+h)^{-1}\ast\nabla_{g}h\ast\nabla f + (g^{-2}\ast h +\sum_{k\geq2}g^{-k-1}\ast h^{k})*\\
&\notag
[\nabla^2f + (g+h)^{-1}\ast\nabla_{g}h\ast\nabla f]
\end{align}
and
\begin{align}
\Delta^2_{g+h}f&= \Delta_{g}(\Delta_{g+h}f) + g^{-1}(g+h)^{-1}\ast\nabla_{g}h\ast\nabla(\Delta_{g+h}f) \\
&\notag
+ (g^{-2}\ast h +\sum_{k\geq2}g^{-k-1}\ast h^{k})[\nabla^2(\Delta_{g+h}f) + (g+h)^{-1}\ast\nabla_{g}h\ast\nabla(\Delta_{g+h}f)]\\
&\notag
=\Delta^2_{g}f + \Delta_{g}(g^{-1}(g+h)^{-1}\ast\nabla_{g}h\ast\nabla f) \\
&\notag
+ \Delta_{g}[(g^{-2}\ast h +\sum_{k\geq2}g^{-k-1}\ast h^{k})(\nabla^2f +(g+h)^{-1}\ast\nabla_{g}h\ast\nabla f)] \\
&\notag
+ g^{-1}(g+h)^{-1}\ast\nabla_{g}h\ast[\nabla(\Delta_{g}f) + \nabla(g^{-1}(g+h)^{-1}\ast\nabla_{g}h\ast\nabla f) \\
&\notag
+ \nabla((g^{-2}\ast h +\sum_{k\geq2}g^{-k-1}\ast h^{k})(\nabla^2f + (g+h)^{-1}\ast\nabla_{g}h\ast\nabla f))] \\
&\notag
+ (g^{-2}\ast h +\sum_{k\geq2}g^{-k-1}\ast h^{k})\{\nabla^2(\Delta f) + \nabla^2( g^{-1}(g+h)^{-1}\ast\nabla_{g}h\ast\nabla f) \\
&\notag
+ \nabla^2 [ (g^{-2}\ast h +\sum_{k\geq2}g^{-k-1}\ast h^{k})(\nabla^2f + (g+h)^{-1}\ast\nabla_{g}h\ast\nabla f)] \\
&\notag
+ (g+h)^{-1}\ast\nabla_{g}h[\nabla(\Delta f) + \nabla( g^{-1}(g+h)^{-1}\ast\nabla_{g}h\ast\nabla f)) \\
&\notag
+ \nabla((g^{-2}\ast h +\sum_{k\geq2}g^{-k-1}\ast h^{k})(\nabla^2f +(g+h)^{-1}\ast\nabla_{g}h\ast\nabla f))]\}\\
&\notag
= \Delta^2_{g}f + O(b^3)\nabla f + O(b^{-1})\nabla f + O(b)\Delta(\nabla f) + O(b^2)\Delta f + O(b)\nabla f \\
&\notag
+ [O(1)+\sum_{k\geq2}O(b^{2k-2})](\nabla^2f + O(b)\nabla f) \\
&\notag
+ [O(b^2)+\sum_{k\geq2}O(b^{2k})](\Delta^2f + O(b^3)\nabla f+O(b^{-1})\nabla f + O(b)\Delta(\nabla f) \\
&\notag
+O(b^2)\Delta f + O(b)\nabla f)\\
&\notag
+ [O(b)+\sum_{k\geq2}O(b^{2k-1})](\nabla^3f + O(b^2)\nabla f+O(1)\nabla f+O(b)\nabla^2f)\\
&\notag
+(O(b)\nabla(\Delta f)+ [O(b^2)+\sum_{k\geq2}O(b^{2k})](\nabla^2f+O(b)\nabla f).
\end{align}
Thus we have
\begin{equation}
|\Delta^2_{g_{a,b}}f - \Delta^2_{g_2}f|\leq O(b^{-1})|\nabla f| + O(b)|\Delta_{g_2}(\nabla f)| + O(1)|\nabla_{g_2}^2f|+O(b^2)|\Delta^2_{g_2}f| + l.o.t.
\end{equation}
For the second block of the right hand side of equation (\ref{eqc1}), we have
\begin{align}
&(g+h)^{-1}\nabla_{g+h}R_{g+h}\nabla_{g+h}f \\
&\notag
= (g+h)^{-1}[\nabla_{g}(R_{g} + (g^{-2}\ast h +\sum_{k\geq2}g^{-k-1}\ast h^{k})\ast\ric_{g} + \delta_{ik}(g+h)^{-2}\ast\nabla^2h \\
&\notag
+\delta_{ik} (g+h)^{-3}\ast\nabla h\ast\nabla h)]\nabla f\\
&\notag
=(g^{-1}+ g^{-2}\ast h +\sum_{k\geq2}g^{-k-1}\ast h^{k})[\nabla_{g}R_{g} + g^{-2}\ast\nabla h\ast\ric_{g}+g^{-2}\ast h\ast\nabla\ric_{g}\\
&\notag
+k\sum_{k\geq2}g^{-k-1}\ast h^{k-1}\ast\nabla h\ast\ric_{g}+ \sum_{k\geq2}g^{-k-1}\ast h^{k}\ast\nabla\ric_{g}+(g+h)^{-3}\ast\nabla h\ast\nabla^2h\\
&\notag
+ (g+h)^{-2}\ast\nabla^3_{g}h + (g+h)^{-4}\ast(\nabla h)^3 + 2(g+h)^{-3}\ast\nabla^2h\ast\nabla h]\nabla f.
\end{align}
Then,
\begin{align}
&|g_{a,b}^{-1}\nabla_{g_{a,b}}R_{g_{a,b}}\nabla_{g_{a,b}}f-g_2^{-1}\nabla_{g_2}R_{g_2}\nabla_{g_2}f|\\
&\notag
\leq \big[O(b)|\ric_{g_2}|+O(b^2)|\nabla_{g_2}\ric_{g_2}|+O(b)+O(b^{-1}) +O(b^3)\\
&\notag
+O(b^2)|\nabla_{g_2}R_{g_2}|\big]|\nabla_{g_2}f|+l.o.t.
\end{align}
For the third block of the right hand side of equation (\ref{eqc1}), we have
\begin{align}
R_{g+h}\Delta_{g+h}f &= R_{g+h}\{\Delta_{g}f + g^{-1}(g+h)^{-1}\ast\nabla_{g}h\ast\nabla f \\
&\notag
+ (g^{-2}\ast h +\sum_{k\geq2}g^{-k-1}\ast h^{k})[\nabla^2f + (g+h)^{-1}\ast\nabla_{g}h\ast\nabla f]\}\\
&\notag
=R_{g}\Delta_{g}f +R_{g} \{g^{-1}(g+h)^{-1}\ast\nabla_{g}h\ast\nabla f \\
&\notag
+ (g^{-2}\ast h +\sum_{k\geq2}g^{-k-1}\ast h^{k})[\nabla^2f + (g+h)^{-1}\ast\nabla_{g}h\ast\nabla f]\}\\
&\notag
+\{(g^{-2}\ast h +\sum_{k\geq2}g^{-k-1}\ast h^{k})\ast\ric_{g} + \delta_{ik}(g+h)^{-2}\ast\nabla^2h \\
&\notag
+\delta_{ik} (g+h)^{-3}\ast\nabla h\ast\nabla h\}\{\Delta_{g}f + g^{-1}(g+h)^{-1}\ast\nabla_{g}h\ast\nabla f \\
&\notag
+ (g^{-2}\ast h +\sum_{k\geq2}g^{-k-1}\ast h^{k})[\nabla^2f + (g+h)^{-1}\ast\nabla_{g}h\ast\nabla f]\}.
\end{align}
Thus,
\begin{align}
|R_{g_{a,b}}\Delta_{g_{a,b}}f-R_{g_2}\Delta_{g_2}f| &\leq O(b)|R_{g_2}||\nabla_{g_2}f| + O(b^2)|R_{g_2}||\nabla_{g_2}^2f| + O(1)|\Delta_{g_2}f| \\
&\notag
+O(b)|\nabla_{g_2}f|+O(b^2)|\nabla_{g_2}^2f|+l.o.t..
\end{align}
For the fourth block of the right hand side of equation (\ref{eqc1}), we calculate
\begin{align}
(g+h)^{-1}\nabla_{g+h}\ric_{g+h}\nabla_{g+h}f &= (g^{-1}+g^{-2}\ast h +\sum_{k\geq2}g^{-k-1}\ast h^{k})\{\nabla_{g}\ric_{g}\\
&\notag
+\nabla_{g}(\delta_{ik}(g+h)^{-1}\ast\nabla^2h + \delta_{ik}(g+h)^{-2}\ast\nabla h\ast\nabla h)\\
&\notag
+ (g+h)^{-1}\ast\nabla_{g}h\ast(\ric_{g}+\delta_{ik}(g+h)^{-1}\ast\nabla^2h \\
&\notag
+ \delta_{ik}(g+h)^{-2}\ast\nabla h\ast\nabla h)\}\nabla f.
\end{align}
This implies
\begin{align}
&|(g_{a,b})^{-1}\nabla_{g_{a,b}}\ric_{g_{a,b}}\nabla_{g_{a,b}}f - g_2^{-1}\nabla_{g_2}\ric_{g_2}\nabla_{g_2}f |\\
&\notag
 \leq O(b^{-1})|\nabla f|+ O(b)|\ric_{g_2}||\nabla f|+ O(b^2)|\nabla_{g_2}\ric_{g_2}||\nabla f| +l.o.t..
\end{align}
For second to last part of the right hand side of equation (\ref{eqc1}), we have
\begin{align}
&(g+h)^{-1}(g+h)^{-1}\ric_{g+h}\nabla_{g+h}^2f \\
&\notag
= (g^{-1}+g^{-2}\ast h +\sum_{k\geq2}g^{-k-1}\ast h^{k})^2\ast\\
&\notag
(\ric_{g}+\delta_{ik}(g+h)^{-1}\ast\nabla^2h + \delta_{ik}(g+h)^{-2}\ast\nabla h\ast\nabla h)(\nabla_{g}^2f +(g+h)^{-1}\ast\nabla_{g}h\ast\nabla f).
\end{align}
Then,
\begin{align}
&|g_{a,b}^{-1}g_{a,b}^{-1}\ric_{g_{a,b}}\nabla_{g_{a,b}}^2f-g_2^{-1}g_2^{-1}\ric_{g_2}\nabla_{g_2}^2f | \\
&\notag
\leq O(b)|\ric_{g_2}||\nabla f| + O(b^2)|\ric_{g_2}||\nabla^2f|+O(1)|\nabla^2f|+O(b)|\nabla f|+ l.o.t.
\end{align}
Suppose $f$ lives in the support of $\beta_2$. Note that in $\text{supp}(\beta_2)$, the weight function $w=|u|=\rho_2$. Combine all the estimates we calculate above, we can estimate the $C^{0}_{\delta-4}(\widetilde{M})$ norm of $L_{g_{a,b}}f - L_{g_2}f$ first, and then we deal with the H\"older part of the weighted norm.
\begin{align}\label{esC0}
&||L_{g_{a,b}}f - L_{g_2}f||_{C^{0}_{\delta-4}(\widetilde{M})}\\
&\notag
= \sup_{\widetilde{M}}|\rho_2^{4-\delta}(L_{g_{a,b}}f - L_{g_2}f)|\\
&\notag
\leq C\sup_{\widetilde{M}}\left( O(b^{-1})|\rho_2^{4-\delta}\nabla f |+O(b)|\rho_2^{4-\delta}\Delta_{g_2}(\nabla_{g_2} f)| +O(1)|\rho_2^{4-\delta}\nabla^2_{g_2}f| \right.\\
&\notag
+O(b^2)|\rho_2^{4-\delta}\Delta_{g_2}^2f| + O(b)|\rho_2^{4-\delta}\ric_{g_2}\nabla_{g_2}f| +O(b^2)|\rho_2^{4-\delta}\nabla_{g_2}\ric_{g_2}\nabla_{g_2}f | \\
&\notag
 +O(b^2)|\rho_2^{4-\delta}\nabla_{g_2}R_{g_2}\nabla_{g_2}f| +O(b)|\rho_2^{4-\delta}R_{g_2}\nabla_{g_2}f| + O(b^2)|\rho_2^{4-\delta}R_{g_2}\nabla_{g_2}^2f| \\
&\notag
\left.+ O(1)|\rho_2^{4-\delta}\Delta_{g_2}f|+ O(b^2)|\rho_2^{4-\delta}\ric_{g_2}\nabla^2f| + O(b^{-2})|\rho_2^{4-\delta}f|\right) \\
&\notag
\leq O(b^2)||f||_{C^{4,\alpha}_{\delta}(\widetilde{M})}+ O(b^4)||f||_{C^{4,\alpha}_{\delta}(\widetilde{M})}
\end{align}
 Now, we deal with the H\"older part in the following.
\begin{align}\label{estH0}
&|L_{g_{a,b}}f - L_{g_2}f|_{C^{0,\alpha}_{\delta-4}(\widetilde{M})}\\
&\notag
\leq C \{\sup_{\widetilde{M}}\left(\rho_2^{4-\delta+\alpha}\sup_{0<4d(x,y)<\rho_2(x)}O(b^2)\frac{|\Delta^2f(x)-\Delta^2f(y)|}{|x-y|^{\alpha}}\right)\\
&\notag
+\sup_{\widetilde{M}}\left(\rho_2^{4-\delta+\alpha}\sup_{0<4d(x,y)<\rho_2(x)}O(b)\frac{|\nabla^3f(x)-\nabla^3f(y)|}{|x-y|^{\alpha}}\right)\\
&\notag
+\sup_{\widetilde{M}}\left(\rho_2^{4-\delta+\alpha}\sup_{0<4d(x,y)<\rho_2(x)}O(1)\frac{|\nabla^2f(x)-\nabla^2f(y)|}{|x-y|^{\alpha}}\right)\\
&\notag
+\sup_{\widetilde{M}}\left(\rho_2^{4-\delta+\alpha}\sup_{0<4d(x,y)<\rho_2(x)}O(b^{-1})\frac{|\nabla f(x)-\nabla f(y)|}{|x-y|^{\alpha}}\right)\\
&\notag
+\sup_{\widetilde{M}}\left(\rho_2^{4-\delta+\alpha}\sup_{0<4d(x,y)<\rho_2(x)}O(b^{-2})\frac{|f(x)-f(y)|}{|x-y|^{\alpha}}\right)\}
\end{align}
All the H\"older estimates are similar. Here, we only pick two terms for demonstration. For example, we have
\begin{align}\label{estH1}
&\sup_{\widetilde{M}}\left(\rho_2^{4-\delta+\alpha}\sup_{0<4d(x,y)<\rho_2(x)}O(b^2)\frac{|\Delta^2f(x)-\Delta^2f(y)|}{|x-y|^{\alpha}}\right)\\
&\notag
\leq O(b^2)|\Delta^2f|_{C^{0,\alpha}_{\delta-4}(\widetilde{M})}\\
&\notag
\leq O(b^2) ||f||_{C^{4,\alpha}_{\delta}(\widetilde{M})}.
\end{align}
and
\begin{align}\label{estH2}
&\sup_{\widetilde{M}}\left(\rho_2^{4-\delta+\alpha}\sup_{0<4d(x,y)<\rho_2(x)}O(b^{-2})\frac{|f(x)-f(y)|}{|x-y|^{\alpha}}\right)\\
&\notag
\leq \sup_{\widetilde{M}}\left(\rho_2^{4-\delta+\alpha}\sup_{0<4d(x,y)<\rho_2(x)}O(b^{-2})\frac{|\nabla f||x-y|}{|x-y|^{\alpha}}\right)\\
&\notag
\leq \sup_{\widetilde{M}}\left(\rho_2^{4-\delta+\alpha}\sup_{0<4d(x,y)<\rho_2(x)}O(b^{-2})|\nabla f||x-y|^{1-\alpha}\right)\\
&\notag
\leq \sup_{\widetilde{M}}\left(\rho_2^{4-\delta+\alpha}\sup_{0<4d(x,y)<\rho_2(x)}O(b^{-2})|\nabla f|||x|+|y||^{1-\alpha}\right)\\
&\notag
\leq C (\sup_{\widetilde{M}}|\rho_2^4O(b^{-2})|)(\sup_{\widetilde{M}}|\rho_2^{1-\delta}|\nabla f||)(\sup_{\widetilde{M}}|\rho_2^{\alpha-1}|x|^{1-\alpha}|)\\
&\notag
\leq O(b^2) ||f||_{C^{4,\alpha}_{\delta}(\widetilde{M})}.
\end{align}
Since the calculation is direct but lengthy, we omit the rest of H\"older estimates. By equation (\ref{esC0}) and above H\"older estimates, we conclude
\begin{equation}
||L_{g_{a,b}}f - L_{g_2}f||_{C^{0,\alpha}_{\delta-4}(\widetilde{M})} = O(b^2) ||f||_{C^{4,\alpha}_{\delta}(\widetilde{M})}= o(1)||f||_{C^{4,\alpha}_{\delta}(\widetilde{M})}
\end{equation}
where $o(1)$ means it is a constant which goes to zero as $b\rightarrow 0$.

Now, let's prove the second part of the lemma. In the support of $\beta_1$, $g_{a,b}=a^2b^2g_{N}+\tilde{h}$, where $|\tilde{h}| =|\theta_2(\frac{|u|}{b})(\eta_2(u)-\tilde{\eta_1}(u) )| =O(b^2)$. Thus, $|\nabla\tilde{h}| = O(b)$, $|\nabla^2\tilde{h}|=O(1)$ and $|\nabla^3\tilde{h}|=O(b^{-1})$, where $\nabla$ is the connection with respect to the metric $a^2b^2g_{N}$. Suppose $f$ lives in the support of $\beta_1$. We then have the similar $C^{0}_{\delta-4}(\widetilde{M})$ estimate as above.
\begin{align}
&||L_{g_{a,b}}f - L_{a^2b^2g_{N}}f||_{C^{0}_{\delta-4}(\widetilde{M})}\\
&\notag
\leq ||\Delta^2_{g_{a,b}}f - \Delta^2_{a^2b^2g_{N}}f|| \\
&\notag
+C\{||g_{a,b}^{-1}\nabla_{g_{a,b}}R_{g_{a,b}}\nabla_{g_{a,b}}f- (a^2b^2g_{N})^{-1}\nabla_{a^2b^2g_{N}}R_{a^2b^2g_{N}}\nabla_{a^2b^2g_{N}}f||\\
&\notag
+||R_{g_{a,b}}\Delta_{g_{a,b}}f-R_{a^2b^2g_{N}}\Delta_{a^2b^2g_{N}}f|| \\
&\notag
+||g_{a,b}^{-1}\nabla_{g_{a,b}}\ric_{g_{a,b}}\nabla_{g_{a,b}}f - (a^2b^2g_{N})^{-1}\nabla_{a^2b^2g_{N}}\ric_{a^2b^2g_{N}}\nabla_{a^2b^2g_{N}}f ||\\
&\notag
+||g_{a,b}^{-1}g_{a,b}^{-1}\ric_{g_{a,b}}\nabla_{g_{a,b}}^2f-(a^2b^2g_{N})^{-1}(a^2b^2g_{N})^{-1}\ric_{a^2b^2g_{N}}\nabla_{a^2b^2g_{N}}^2f || \\
&\notag
+ ||(Q^{n}_{g_{a,b}}-Q^{n}_{a^2b^2g_{N}})f||\}\\
&\notag
\leq C\sup_{\widetilde{M}}( O(b^{-1})|w^{4-\delta}\nabla f |+O(b)|w^{4-\delta}\Delta(\nabla f)| + O(1)|w^{4-\delta}\nabla^2f|+O(b^2)|w^{4-\delta}\Delta^2f| \\
&\notag
+ O(b)|w^{4-\delta}\ric_{a^2b^2g_{N}}\nabla_{a^2b^2g_{N}}f| +O(b^2)|w^{4-\delta}\nabla_{a^2b^2g_{N}}\ric_{a^2b^2g_{N}}\nabla_{a^2b^2g_{N}}f|\\
&\notag
+O(b^2)|w^{4-\delta}\nabla_{a^2b^2g_{N}}R_{a^2b^2g_{N}}\nabla_{a^2b^2g_{N}}f| +O(b)|w^{4-\delta}R_{a^2b^2g_{N}}\nabla_{a^2b^2g_{N}}f| \\
&\notag
+ O(b^2)|w^{4-\delta}R_{a^2b^2g_{N}}\nabla_{a^2b^2g_{N}}^2f|+ O(b^2)|w^{4-\delta}\ric_{a^2b^2g_{N}}\nabla^2f| +O(b^{-2})|w^{4-\delta}f|) \\
&\notag
\leq (ab)^{4-\delta}\sup_{\widetilde{M}}\{O(b^{-1})a^{-3}(ab)^{-1}|\rho_1^{1-\delta}\nabla_{g_{N}}f| + O(b)a^{-1}(ab)^{-3}|\rho_1^{3-\delta}\Delta_{g_{N}}\nabla f| \\
&\notag
+ a^{-2}(ab)^{-2}|\rho_1^{2-\delta}\nabla^2_{g_{N}}f| +O(b^2)(ab)^{-4}|\rho_1^{4-\delta}\Delta^2_{g_{N}}f| + O(b^3)a^{-3}(ab)^{-1}|\rho_1^{1-\delta}\nabla_{g_{N}}f| \\
&\notag
+ O(b^4)a^{-2}(ab)^{-2}|\rho_1^{2-\delta}\nabla^2_{g_{N}}f| +O(b^{-2})a^{-4}|\rho_1^{-\delta}f| \}\\
&\notag
\leq(ab)^{-\delta}(O(b^2)+O(b^6))|| f ||_{C^{4,\alpha}_{\delta, g_{N}}(N_{p})}\\
&\notag
= (ab)^{-\delta}(O(b^2)+O(b^6)) (ab)^{\delta}|| f ||_{C^{4,\alpha}_{\delta}(\widetilde{M})}\\
&\notag
= (O(b^2) + O(b^6))||f||_{C^{4,\alpha}_{\delta}(\widetilde{M})}
\end{align}
The H\"older parts of the weighted norm are estimated by an argument similar to equations (\ref{estH0}), (\ref{estH1}) and (\ref{estH2}).  Thus we conclude that
\begin{equation}
||L_{g_{a,b}}f - L_{a^2b^2g_{N}}f||_{C^{0,\alpha}_{\delta-4}(\widetilde{M})}= O(b^2)||f||_{C^{4,\alpha}_{\delta}(\widetilde{M})} = o(1)||f||_{C^{4,\alpha}_{\delta}(\widetilde{M})},
\end{equation}
where o(1) means it is a constant which goes to zero as $b\rightarrow 0$. We then complete the proof of the lemma.
\end{proof}

With Lemma \ref{L1}, we only need to estimate commutator terms $[L_{a^2b^2g_{N}}, \beta_1]$ and $[L_{g_2}, \beta_2]$.
\begin{align}\label{eqcm1}
&[L_{a^2b^2g_{N}},\beta_1]f \\
&\notag
= (L_{a^2b^2g_{N}}\beta_1 - \beta_1L_{a^2b^2g_{N}}) f \\
&\notag
= (a^2b^2g_{N})^{-1}\nabla\beta_1\nabla(\Delta f) + \Delta\beta_1\Delta f + (a^2b^2g_{N})^{-1}\nabla(\Delta\beta_1)\nabla f + f\Delta^2\beta_1 \\
&\notag
- div_{g_{a^2b^2g_{N}}}\left[\left(b_{n}R_{a^2b^2g_{N}}(a^2b^2g_{N})- \frac{4}{n-2}\ric_{a^2b^2g_{N}}\right)(d\beta_1)f\right],
\end{align}
where $\nabla$ and $\Delta$ are the connection and Laplacian with respect to the metric $a^2b^2g_{N}$ respectively and $b_{n}$ is a dimensional constant.

On the support of $\nabla\beta_1$, since $a^2b^2g_{N}$ is close to the flat metric $(du)^2$, we have the following estimates:
\begin{equation}\label{esb1}
|\nabla\beta_1| =\left| \theta_1\rq{}\lambda\left(\frac{|u|}{4b}\right)^{\lambda-1}(4b)^{-1}\right|= O(\lambda b^{-1}),
\end{equation}
\begin{align}
|\nabla^2\beta_1|
&\notag
=\left| \theta_1\rq{}\rq{}\lambda^2\left(\frac{|u|}{4b}\right)^{2\lambda-2}(4b)^{-2}+\theta_1\rq{}\lambda(\lambda-1)\left(\frac{|u|}{4b}\right)^{\lambda-2}(4b)^{-2}\right| \\
&\notag
= O(\lambda b^{-2})+ O(\lambda^2b^{-2}),
\end{align}
\begin{equation}
|\nabla^3\beta_1| = O(\lambda^3b^{-3}) + O(\lambda^2b^{-3})+ O(\lambda b^{-3}),
\notag
\end{equation}
and
\begin{equation}
|\Delta^2\beta_1| = O(\lambda^4b^{-4}) + O(\lambda^3b^{-4})+ O(\lambda^2b^{-4})+ O(\lambda b^{-4}).
\notag
\end{equation}
The function $\gamma_1\phi$ can be thought of as a function on $N_{p}$. Since $L_{g_N}$ has uniformly controlled inverse $H_1$ and $||\gamma_1\phi||_{C^{0,\alpha}_{\delta-4,g_{N}}(N_{p})}\leq C(ab)^{\delta-4}||\phi||_{C^{0,\alpha}_{\delta-4}(\widetilde{M})}$, we have
\begin{align}
&||a^4b^4H_1\gamma_1\phi||_{C^{4,\alpha}_{\delta}(\widetilde{M})}\\
&\notag
= a^4b^4(ab)^{-\delta}||H_1\gamma_1\phi||_{C^{4,\alpha}_{\delta, g_{N}}(N_{p})}\\
&\notag
\leq Ca^4b^4(ab)^{-\delta}||\gamma_1\phi||_{C^{0,\alpha}_{\delta-4,g_{N}}(N_{p})}\\
&\notag
\leq Ca^4b^4(ab)^{-\delta}(ab)^{\delta-4}||\phi||_{C^{0,\alpha}_{\delta-4}(\widetilde{M})}\\
&\notag
= C||\phi||_{C^{0,\alpha}_{\delta-4}(\widetilde{M})},
\end{align}
where $C$ is a uniform constant independent of $a$ and $b$.

For convenience, denote $f = a^4b^4H_1\gamma_1\phi$. By equation (\ref{eqcm1}) and estimates from (\ref{esb1}), we have
\begin{align}
&||[L_{a^2b^2g_{N}},\beta_1]a^4b^4H_1\gamma_1\phi||_{C^{0}_{\delta-4}(\widetilde{M})}\\
&\notag
\leq \sup_{\widetilde{M}}(O(\lambda b^{-1})|w^{4-\delta}\nabla(\Delta f)| + O(\lambda b^{-2})|w^{4-\delta}\Delta f| + O(\lambda b^{-3})|w^{4-\delta}\nabla f| \\
&\notag
+ O(\lambda b^{-4})|w^{4-\delta}f| + |w^{4-\delta}R_{a^2b^2g_{N}}\nabla((d\beta_1)f)| + |w^{4-\delta}\nabla R_{a^2b^2g_{N}}(d\beta_1)f| \\
&\notag
+ |w^{4-\delta}\ric_{a^2b^2g_{N}}\nabla((d\beta_1)f)| + |w^{4-\delta}\nabla\ric_{a^2b^2g_{N}}(d\beta_1)f| )\\
&\notag
\leq (ab)^{4-\delta}\sup_{\widetilde{M}}\{O(\lambda b^{-1})a^{-1}(ab)^{-3}|\rho_1^{3-\delta}\nabla_{g_{N}}(\Delta_{g_{N}}f)| + O(\lambda b^{-2})a^{-2}(ab)^{-2}|\rho_1^{2-\delta}\Delta_{g_{N}}f| \\
&\notag
+ O(\lambda b^{-3})a^{-3}(ab)^{-1}|\rho_1^{1-\delta}\nabla_{g_{N}}f| +O(\lambda b^{-4})a^{-4}|\rho_1^{-\delta}f| + O(\lambda)a^{-4}|\rho_1^{-\delta}f| \\
&\notag
+ O(\lambda b)a^{-3}(ab)^{-1}|\rho_1^{1-\delta}\nabla_{g_{N}} f| \} \\
&\notag
= (ab)^{-\delta}(O(\lambda)+O(\lambda b^4))||f||_{C^{4,\alpha}_{\delta, g_{N}}(N_{p})}\\
&\notag
= (ab)^{-\delta}(O(\lambda)+O(\lambda b^4)) (ab)^{\delta}||f||_{C^{4,\alpha}_{\delta}(\widetilde{M})}\\
&\notag
= (O(\lambda)+O(\lambda b^4)) ||f||_{C^{4,\alpha}_{\delta}(\widetilde{M})}\\
&\notag
\leq C(O(\lambda)+O(\lambda b^4))||\phi||_{C^{0,\alpha}_{\delta-4}(\widetilde{M})}\\
&\notag
=o(1)||\phi||_{C^{0,\alpha}_{\delta-4}(\widetilde{M})},
\end{align}
where $\lambda = O(\frac{1}{\log b})$ as we mentioned earlier. For the H\"older part, we have the following estimates.
\begin{align}\label{estH3}
&|[L_{a^2b^2g_{N}},\beta_1]a^4b^4H_1\gamma_1\phi|_{C^{0,\alpha}_{\delta-4}(\widetilde{M})}\\
&\notag
\leq |O(\lambda b^{-1})\nabla^3f|_{C^{0,\alpha}_{\delta-4}(\widetilde{M})} + |O(\lambda b^{-2})\nabla^2f|_{C^{0,\alpha}_{\delta-4}(\widetilde{M})} + |O(\lambda b^{-3})\nabla f|_{C^{0,\alpha}_{\delta-4}(\widetilde{M})}\\
&\notag
+ |O(\lambda b^{-4})f|_{C^{0,\alpha}_{\delta-4}(\widetilde{M})} + |O(\lambda)f|_{C^{0,\alpha}_{\delta-4}(\widetilde{M})} + |O(\lambda b)\nabla f|_{C^{0,\alpha}_{\delta-4}(\widetilde{M})}
\end{align}
For $|O(\lambda b^{-1})\nabla^3f|_{C^{0,\alpha}_{\delta-4}(\widetilde{M})}$ term, we have the Ho\"lder estimate like
\begin{align}\label{estH4}
&\sup_{\widetilde{M}}\left(w^{4-\delta+\alpha}\sup_{0<4d(x,y)<w(x)}O(\lambda b^{-1})\frac{|\nabla^3f(x)-\nabla^3f(y)|}{d(x,y)^{\alpha}}\right)\\
&\notag
\leq C\sup_{\widetilde{M}}\left(w^{4-\delta+\alpha}\sup_{0<4d(x,y)<w(x)}O(\lambda b^{-1})\frac{|\nabla^4f||x-y|}{|x-y|^{\alpha}}\right)\\
&\notag
\leq \sup_{\widetilde{M}}\left((ab\rho_1)^{4-\delta+\alpha}\sup_{0<4d(z,\zeta)<\rho_1(z)}O(\lambda b^{-1})O(a^{-3-\alpha}b^{-3-\alpha})|\nabla^4_{g_{N}} f||z-\zeta|^{1-\alpha}\right)\\
&\notag
\leq (ab)^{-\delta}\sup_{\widetilde{M}}\left(O(ab)\rho_1^{4-\delta+\alpha}\sup_{0<4d(z,\zeta)<\rho_1(z)}O(\lambda b^{-1})|\nabla^4_{g_{N}} f|||z|+|\zeta||^{1-\alpha}\right)\\
&\notag
\leq C(ab)^{-\delta}O(\lambda)(\sup_{\widetilde{M}}|\rho_1^{4-\delta}|\nabla^4_{g_{N}}f||)(O(a)\sup_{\widetilde{M}}|\rho_1^{\alpha}|z|^{1-\alpha}|)\\
&\notag
\leq (ab)^{-\delta}O(\lambda) ||f||_{C^{4,\alpha}_{\delta,g_{N}}(N_{p})}\\
&\notag
= (ab)^{-\delta}O(\lambda) (ab)^{\delta}||f||_{C^{4,\alpha}_{\delta}(\widetilde{M})}\\
&\notag
\leq O(\lambda)||\phi||_{C^{0,\alpha}_{\delta-4}(\widetilde{M})}.
\end{align}
The rest of the terms in equation (\ref{estH3}) are estimated similarly as the $|O(\lambda b^{-1})\nabla^3f|$ term. Thus we know
\begin{equation}
||[L_{a^2b^2g_{N}},\beta_1]a^4b^4H_1\gamma_1\phi||_{C^{0,\alpha}_{\delta-4}(\widetilde{M})}= O(\lambda)||\phi||_{C^{0,\alpha}_{\delta-4}(\widetilde{M})}= o(1)||\phi||_{C^{0,\alpha}_{\delta-4}(\widetilde{M})}.
\end{equation}

Similarly, we have
\begin{align}
[L_{g_2},\beta_2]f &= g^{-1}_2\nabla\beta_2\nabla(\Delta f) + \Delta\beta_2\Delta f + g^{-1}_2\nabla(\Delta\beta_2)\nabla f+ f\Delta^2\beta_2 \\
&\notag
- div_{g_2}\left[\left(b_{n}R_{g_2}g_2- \frac{4}{n-2}\ric_{g_2}\right)(d\beta_2)f\right],
\end{align}
where $\nabla$ and $\Delta$ are the connection and Laplacian with respect to the metric $g_2$ respectively and $b_{n}$ is a dimensional constant. On the support of $\beta_2$, we also have similar estimates.
\begin{equation}
|\nabla\beta_2| = \left|\theta_2\rq{}\lambda\left(\frac{4|u|}{b}\right)^{\lambda-1}4(b)^{-1}\right|= O(\lambda b^{-1}),
\end{equation}
\begin{align}
|\nabla^2\beta_2|
&\notag
= \left|\theta_2\rq{}\rq{}\lambda^2\left(\frac{4|u|}{b}\right)^{2\lambda-2}4^2(b)^{-2}+\theta_2\rq{}\lambda(\lambda-1)\left(\frac{4|u|}{b}\right)^{\lambda-2}4^2(b)^{-2}\right| \\
&\notag
= O(\lambda b^{-2})+ O(\lambda^2b^{-2}),
\end{align}
\begin{equation}
|\nabla^3\beta_2| = O(\lambda^3b^{-3}) + O(\lambda^2b^{-3})+ O(\lambda b^{-3}),
\notag
\end{equation}
and
\begin{equation}
|\Delta^2\beta_2| = O(\lambda^4b^{-4}) + O(\lambda^3b^{-4})+ O(\lambda^2b^{-4})+ O(\lambda b^{-4}).
\notag
\end{equation}
We can also think of $\gamma_2\phi$ as a function on $M_{q}$. Since we have uniformly controlled inverse $H_2$ of $L_{g_2}$ and $g_{a,b}$ is close to $g_2$, we then have
\begin{equation}
\notag
||H_2\gamma_2\phi||_{C^{4,\alpha}_{\delta,g_2}(M_{q})}\leq C||\gamma_2\phi||_{C^{0,\alpha}_{\delta-4,g_2}(M_{q})}\leq C||\phi||_{C^{0,\alpha}_{\delta-4}(\widetilde{M})}.
\end{equation}
Denote $f = H_2\gamma_2\phi$. Note that the weight function $w=|u|=\rho_2$ in the support of $\gamma_2$. Thus, we estimate the $C^{0}_{\delta-4}(\widetilde{M})$ norm in the following.
\begin{align}
&||[L_{g_2},\beta_2]H_2\gamma_2\phi||_{C^{0}_{\delta-4}(\widetilde{M})}\\
&\notag
\leq \sup_{\widetilde{M}}(O(\lambda b^{-1})|\rho^{4-\delta}_2\nabla(\Delta f)| + O(\lambda b^{-2})|\rho^{4-\delta}_2\Delta f| + O(\lambda b^{-3})|\rho^{4-\delta}_2\nabla f| \\
&\notag
+ O(\lambda b^{-4})|\rho^{4-\delta}_2f| + |\rho^{4-\delta}_2R_{g_2}\nabla((d\beta_2)f)| + |\rho^{4-\delta}_2\nabla R_{g_2}(d\beta_2)f| \\
&\notag
+ |\rho^{4-\delta}_2\ric_{g_2}\nabla((d\beta_2)f)| + |\rho^{4-\delta}_2\nabla\ric_{g_2}(d\beta_2)f| )\\
&\notag
\leq\sup_{\widetilde{M}}(O(\lambda)|\rho^{3-\delta}_2\nabla(\Delta f)| + O(\lambda)|\rho^{2-\delta}_2\Delta f| + O(\lambda)|\rho^{1-\delta}_2\nabla f| \\
&\notag
+O(\lambda)|\rho^{-\delta}_2f| + O(\lambda b^2)|R_{g_2}\rho^{-\delta}_2f| + O(\lambda b^2)|R_{g_2}\rho^{1-\delta}_2\nabla f| \\
&\notag
+ O(\lambda b^3)|\nabla R_{g_2}\rho^{-\delta}_2f| + O(\lambda b^2)|\ric_{g_2}\rho^{-\delta}_2f| + O(\lambda b^2)|\ric_{g_2}\rho^{1-\delta}_2\nabla f| \\
&\notag
+ O(\lambda b^3)|\nabla\ric_{g_2}\rho^{-\delta}_2f| )\\
&\notag
\leq (O(\lambda)+O(\lambda b^2))||f||_{C^{4,\alpha}_{\delta,g_2}(M_{q})}\\
&\notag
\leq (O(\lambda)+O(\lambda b^2))||\phi||_{C^{0,\alpha}_{\delta-4}(\widetilde{M})}
\end{align}
To estimate the H\"older part of the weighted norm, we do the similar estimates as equations (\ref{estH3}) and (\ref{estH4}). We then conclude that
\begin{equation}
||[L_{g_2},\beta_2]H_2\gamma_2\phi||_{C^{0,\alpha}_{\delta-4}(\widetilde{M})}= O(\lambda)||\phi||_{C^{0,\alpha}_{\delta-4}(\widetilde{M})}=o(1)||\phi||_{C^{0,\alpha}_{\delta-4}(\widetilde{M})}.
\end{equation}
Now we can continue to prove Theorem \ref{Thm2}. By applying Lemma \ref{L1} and above estimates, we have
\begin{align}
L_{g_{a,b}}(H_0) &= L_{g_{a,b}}(a^4b^4\beta_1H_1\gamma_1) + L_{g_{a,b}}(\beta_2H_2\gamma_2)\\
&\notag
=\beta_1L_{g_{a,b}}(a^4b^4H_1\gamma_1) + \beta_2L_{g_{a,b}}(H_2\gamma_2) \\
&\notag
+ [L_{g_{a,b}},\beta_1]a^4b^4H_1\gamma_1+[L_{g_{a,b}},\beta_2]H_2\gamma_2\\
&\notag
=\beta_1(L_{a^2b^2g_{N}}a^4b^4H_1)\gamma_1 +o(1)+ \beta_2(L_{g_2}H_2)\gamma_2 + o(1) +o(1)\\
&\notag
=\beta_1\gamma_1 +\beta_2\gamma_2 + o(1)\\
&\notag
=\gamma_1 +\gamma_2 +o(1)\\
&\notag
=1+R,
\end{align}
where $||R||\leq\frac{1}{2}$ provided $a=b^4$ and $b$ small enough. Thus we know that $H=H_0(1+R)^{-1}$ is the controlled inverse of $L_{g_{a,b}}$ and $||H||\leq 2$.
\end{proof}

\subsection{Proof of Theorem \ref{Thm1}.}\label{IFT}
After we showed the invertibility of the linearized operator $L_{g_{a,b}}$, we can continue to prove Theorem \ref{Thm1}.

Recall the equation we want to solve is
\begin{equation}
\notag
\mathbf{N}_{g_{a,b}}[1+\phi]:=(1+\phi)^{-\frac{n+4}{n-4}}P^{n}_{g_{a,b}}[1+\phi]-\frac{n-4}{2}\nu=0.
\end{equation}
Then we can rewrite the equation as
\begin{equation}
\mathbf{N}_{g_{a,b}}[1+\phi]=\mathbf{N}_{g_{a,b}}[1]+ L_{g_{a,b}}[\phi]+ \mathbf{q}_{g_{a,b}}[\phi],
\end{equation}
where $\mathbf{N}_{g_{a,b}}[1]= \frac{n-4}{2}Q^{n}_{g_{a,b}}- \frac{n-4}{2}\nu$.
Thus we have the quadratic term
\begin{equation}
\mathbf{q}_{g_{a,b}}[\phi]:= \mathbf{N}_{g_{a,b}}[1+\phi] - \mathbf{N}_{g_{a,b}}[1]-(P^{n}_{g_{a,b}}[\phi]-\frac{n+4}{2}Q^{n}_{g_{a,b}}\phi)
\end{equation}
and the following estimate.

\begin{lemma}\label{lemma2}
Suppose $||\phi||_{C^{4,\alpha}_0(\widetilde{M})}$, $||\psi||_{C^{4,\alpha}_0(\widetilde{M})}\leq\epsilon$, where $\epsilon>0$ is a sufficiently small constant. Then
\begin{equation}
||\mathbf{q}_{g_{a,b}}[\psi] - \mathbf{q}_{g_{a,b}}[\phi]||_{C^{0,\alpha}_{\delta-4}(\widetilde{M})}\leq C(||\phi||_{C^{4,\alpha}_0(\widetilde{M})}+||\psi||_{C^{4,\alpha}_0(\widetilde{M})})||\psi -\phi||_{C^{4,\alpha}_{\delta}(\widetilde{M})}
\end{equation}
\end{lemma}
\begin{proof}
First, we estimate the $C^0_{\delta-4}(\widetilde{M})$ norm. We have
\begin{align}
&||\mathbf{q}_{g_{a,b}}[\phi] - \mathbf{q}_{g_{a,b}}[\psi]||_{C^{0}_{\delta-4}(\widetilde{M})}\\
&\notag
\leq C\sup_{\widetilde{M}}\left\{\left|w^{4-\delta}((1+\psi)^{-\frac{n+4}{n-4}}-1)P^{n}_{g_{a,b}}[\psi-\phi]| \right.\right.\\
&\notag
\left.\left. + |w^{4-\delta}P^{n}_{g_{a,b}}[\phi]((1+\psi)^{-\frac{n+4}{n-4}}- (1+\phi)^{-\frac{n+4}{n-4}})\right| \right.\\
&\notag
\left.+ \left|w^{4-\delta}Q^{n}_{g_{a,b}}\left[((1+\psi)^{-\frac{n+4}{n-4}}+\frac{n+4}{n-4}\psi- (1+\phi)^{-\frac{n+4}{n-4}}-\frac{n+4}{n-4}\phi\right]\right|\right\}\\
&\notag
\leq C\sup_{\widetilde{M}}\big\{\big|w^{4-\delta}(|\psi|+O(|\psi|^2))P^{n}_{g_{a,b}}[\psi-\phi]\big| \\
&\notag
+ \big|w^{4-\delta}P^{n}_{g_{a,b}}[\phi](|\psi-\phi|+O(|\psi^2-\phi^2|))\big| \\
&\notag
+ \big|w^{4-\delta}Q^{n}_{g_{a,b}}(|\psi^2-\phi^2|+O(|\psi^3-\phi^3|))\big|\big\} \\
&\notag
\leq C\sup_{\widetilde{M}}\big\{|\psi||w^{4-\delta}P^{n}_{g_{a,b}}[\psi-\phi]| + (w^{-\delta}|\psi-\phi||w^{4}P^{n}_{g_{a,b}}[\phi]|)\\
&\notag
+ |w^{4-\delta}|\psi+\phi||\psi-\phi|Q^{n}_{g_{a,b}}|\big\}\\
&\notag
\leq C\big\{|\psi|_{C^0}||\psi - \phi||_{C^{4,\alpha}_{\delta}(\widetilde{M})}+ |\phi|_{C^{4,\alpha}_0(\widetilde{M})}||\psi - \phi||_{C^{4,\alpha}_{\delta}(\widetilde{M})} \\
&\notag
+||Q_{g_{a,b}}||_{C^{0,\alpha}_{-4}(M)}(|\phi|_{C^0}+|\psi|_{C^0})||\psi - \phi||_{C^{4,\alpha}_{\delta}(\widetilde{M})}\big\}\\
&\notag
\leq C(|\phi|_{C^{4,\alpha}_0(\widetilde{M})}+|\psi|_{C^{4,\alpha}_0(\widetilde{M})})||\psi - \phi||_{C^{4,\alpha}_{\delta}(\widetilde{M})},
\end{align}
where $||Q^{n}_{g_{a,b}}||_{C^{0,\alpha}_{-4}(\widetilde{M})}=O(1)$. Now we deal with the H\"older part.
\begin{align}\label{estH5}
&|\mathbf{q}_{g_{a,b}}[\phi] - \mathbf{q}_{g_{a,b}}[\psi]|_{C^{0,\alpha}_{\delta-4}(\widetilde{M})} \\
&\notag
\leq |((1+\psi)^{-\frac{n+4}{n-4}}-1)P^{n}_{g_{a,b}}[\psi-\phi]|_{C^{0,\alpha}_{\delta-4}(\widetilde{M})} \\
&\notag
+ \left|P^{n}_{g_{a,b}}[\phi]((1+\psi)^{-\frac{n+4}{n-4}}- (1+\phi)^{-\frac{n+4}{n-4}})\right|_{C^{0,\alpha}_{\delta-4}(\widetilde{M})}\\
&\notag
+ \left|Q^{n}_{g_{a,b}}\left[((1+\psi)^{-\frac{n+4}{n-4}}+\frac{n+4}{n-4}\psi- (1+\phi)^{-\frac{n+4}{n-4}}-\frac{n+4}{n-4}\phi\right]\right|_{C^{0,\alpha}_{\delta-4}(\widetilde{M})}
\end{align}
The H\"older parts for each term in the equation (\ref{estH5}) are estimated as in equation (\ref{estH4}). For example,
\begin{align}
&|((1+\psi)^{-\frac{n+4}{n-4}}-1)P^{n}_{g_{a,b}}[\psi-\phi]|_{C^{0,\alpha}_{\delta-4}(\widetilde{M})}\\
&\notag
\leq C|\psi(x)|_{C^0}\left( \sup_{\widetilde{M}}w^{4-\delta+\alpha}\sup_{0<4d(x,y)< w(x)} \frac{|P^{n}_{g_{a,b}}[\psi-\phi](x) - P^{n}_{g_{a,b}}[\psi-\phi](y)|}{|x-y|^{\alpha}}\right)\\
&\notag
+\left( \sup_{\widetilde{M}}w^{4-\delta+\alpha}\sup_{0<4d(x,y)< w(x)} |P^{n}_{g_{a,b}}[\psi-\phi](y)| \frac{|(\psi(x)-\psi(y) +O(\psi^2(x) - \psi^2(y)))|}{|x-y|^{\alpha}}\right)\\
&\notag
\leq C\{|\psi|_{C^0}|P^{n}_{g_{a,b}}[\psi-\phi]|_{C^{0,\alpha}_{\delta-4}(\widetilde{M})} \\
&\notag
+ (\sup_{\widetilde{M}}w|\nabla\psi|) (\sup_{\widetilde{M}}w^{\alpha-1}\sup_{0<4d(x,y)< w(x)}\frac{|x-y|}{|x-y|^{\alpha}})(\sup_{\widetilde{M}}w^{4-\delta}|P^{n}_{g_{a,b}}[\psi-\phi](y)|)\}\\
&\notag
\leq C ||\psi||_{C^{4,\alpha}_0(\widetilde{M})}||\psi-\phi||_{C^{4,\alpha}_{\delta}(\widetilde{M})}
\end{align}
All the other terms are estimated in a similar way. Thus we proved
\begin{equation}
||\mathbf{q}_{g_{a,b}}[\psi] - \mathbf{q}_{g_{a,b}}[\phi]||_{C^{0,\alpha}_{\delta-4}(\widetilde{M})}\leq C(||\phi||_{C^{4,\alpha}_0(\widetilde{M})}+||\psi||_{C^{4,\alpha}_0(\widetilde{M})})||\psi -\phi||_{C^{4,\alpha}_{\delta}(\widetilde{M})}
\end{equation}
\end{proof}

With the inverse operator $L^{-1}_{g_{a,b}}$ obtained in Theorem \ref{Thm2}, we define a new operator $T_{g_{a,b}}: C^{4,\alpha}_{\delta}(\widetilde{M}) \rightarrow C^{4,\alpha}_{\delta}(\widetilde{M})$ by
\begin{equation}
T_{g_{a,b}}[\phi] := -L^{-1}_{g_{a,b}}\left[\left(\frac{n-4}{2}Q^{n}_{ g_{a,b}}- \frac{n-4}{2}\nu \right)+\mathbf{q}_{g_{a,b}}[\phi]\right].
\end{equation}
By direct calculation, it's clear that solving the equation (\ref{eqN})
$$\mathbf{N}_{g_{a,b}}[1+\phi]=\mathbf{N}_{g_{a,b}}[1]+ L_{g_{a,b}}[\phi]+ \mathbf{q}_{g_{a,b}}[\phi]= 0$$
is the same as showing $\phi$ is a fixed point of $T_{g_{a,b}}$. Furthermore, we want $1+\phi >0$. Thus, the proof of Theorem \ref{Thm1} follows by the following lemma.

\begin{lemma}\label{lemma3}
Suppose $4-n<\delta<0$ and $\delta$ is close enough to $0$. Then $T_{g_{a,b}}$ is a contraction map on the small ball $U:=\{\phi\in\ C^{4,\alpha}_{\delta, g_{a,b}}(\widetilde{M}) : ||\phi||_{C^{4,\alpha}_{\delta}(\widetilde{M})}<b^{1-5\delta}\}$ if $a=b^4$ and $b$ is small enough. In particular, $1+\phi >0$, which gives a constant $Q$-curvature metric in the conformal class of $g_{a,b}$.
\end{lemma}
\begin{proof}
First, we show that $T_{g_{a,b}}$ maps $U$ into $U$. If $a=b^4$, $\phi\in U$ implies
$\sup_{\widetilde{M}}|w^{-\delta}\phi|\leq Cb^{1-5\delta}$. Since $\delta<0$, we have
\begin{equation}
\sup_{\widetilde{M}}|\phi|\leq C\sup_{\widetilde{M}}(b^{1-5\delta}w^{\delta})\leq Cb^{1-5\delta}(ab)^{\delta}=Cb.
\end{equation}
Similarly, we also have
\begin{equation}
\sup_{\widetilde{M}}w^{k}|\nabla^{k}\phi|\leq C\sup_{\widetilde{M}}(b^{1-5\delta}w^{\delta})\leq Cb^{1-5\delta}(ab)^{\delta}=Cb,
\end{equation}
for $1\leq k\leq4$, and
\begin{equation}
\sup_{\widetilde{M}}w^{4+\alpha}\sup_{0<4d(x,y)<w(x)}\frac{|\nabla^4\phi(x)-\nabla^4\phi(y)|}{d(x,y)^{\alpha}}\leq C\sup_{\widetilde{M}}(b^{1-5\delta}w^{\delta})\leq Cb^{1-5\delta}(ab)^{\delta}=Cb.
\end{equation}
Thus, we have $||\phi||_{C^{4,\alpha}_0(\widetilde{M})}=O(b)$. By Lemma \ref{lemma2} and the fact that we have controlled inverse operator $L^{-1}_{g_{a,b}}$, we have
\begin{align}
\big|\big|T_{g_{a,b}}[\phi]\big|\big|_{C^{4,\alpha}_{\delta}(\widetilde{M})}&=\left|\left|-L^{-1}_{g_{a,b}}\left[\left(\frac{n-4}{2}Q^{n}_{ g_{a,b}}- \frac{n-4}{2}\nu \right)+\mathbf{q}[\phi]\right]\right|\right|_{C^{4,\alpha}_{\delta}(\widetilde{M})}\\
&\notag
\leq C\left|\left|\left(\frac{n-4}{2}Q^{n}_{ g_{a,b}} - \frac{n-4}{2}\nu\right)+\mathbf{q}[\phi]\right|\right|_{C^{0,\alpha}_{\delta-4}(\widetilde{M})}\\
&\notag
\leq C\big|\big|\frac{n-4}{2}Q^{n}_{ g_{a,b}}-\frac{n-4}{2}\nu \big|\big|_{C^{0,\alpha}_{\delta-4}(\widetilde{M})}+C||\mathbf{q}[\phi]||_{C^{0,\alpha}_{\delta-4}(\widetilde{M})}\\
&\notag
\leq C\big|\big|Q^{n}_{ g_{a,b}}-\nu \big|\big|_{C^{0,\alpha}_{\delta-4}(\widetilde{M})} +||\phi||_{C^{4,\alpha}_0(\widetilde{M})}||\phi||_{C^{0,\alpha}_{\delta-4}(\widetilde{M})}\\
&\notag
\leq O(b^{2-\delta}) +O(b)O(b^{1-5\delta})\\
&\notag
\leq O(b^{1-5\delta}).
\end{align}
The last inequality holds provided $\delta$ is close to $0$.

Now, we proceed to show $T_{g_{a,b}}$ is a contraction map on $U$ with contraction constant $\frac{1}{2}$. Assume $\psi$, $\phi\in U$, we have $||\psi||_{C^{4,\alpha}_0(\widetilde{M})}$, $||\phi||_{C^{4,\alpha}_0(\widetilde{M})}\leq Cb$. By the estimate from Lemma \ref{lemma2} and the fact that we have an uniform bound for $L^{-1}_{g_{a,b}}$, we know
\begin{align}
||T_{g_{a,b}}[\phi] - T_{g_{a,b}}[\psi]||_{C^{4,\alpha}_{\delta}(\widetilde{M})}&= ||L^{-1}_{g_{a,b}}(\mathbf{q}_{g_{a,b}}[\psi]-\mathbf{q}_{g_{a,b}}[\phi])||_{C^{4,\alpha}_{\delta}(\widetilde{M})}\\
&\notag
\leq C||\mathbf{q}_{g_{a,b}}[\psi]-\mathbf{q}_{g_{a,b}}[\phi]||_{C^{0,\alpha}_{\delta-4}(\widetilde{M})}\\
&\notag
\leq C(||\psi||_{C^{4,\alpha}_0(\widetilde{M})} +||\phi||_{C^{4,\alpha}_0(\widetilde{M})})||\psi-\phi||_{C^{0,\alpha}_{\delta-4}(\widetilde{M})}\\
&\notag
\leq Cb||\psi-\phi||_{C^{0,\alpha}_{\delta-4}(\widetilde{M})}\\
&\notag
<\frac{1}{2}||\psi-\phi||_{C^{0,\alpha}_{\delta-4}(\widetilde{M})},
\end{align}
provided $b$ is small enough, $a=b^4$ and $\delta$ is chosen to be close to $0$.

Once we prove the map $T_{g_{a,b}}$ is a contraction map on the small ball $U$, we conclude that $T_{g_{a,b}}$ has a fixed point $\phi$ in $U$ by Banach Fixed Point Theorem. In particular, from the above argument, $\phi\in U$ implies $|\phi|_{C^0}=O(b)$. Thus, if we choose $b$ small enough, then $1+\phi >0$, which is the conformal factor. If we define the metric $\tilde{g}= (1+\phi)^{\frac{4}{n-4}}g_{a,b}$, $\tilde{g}$ has constant $Q$-curvature.
\end{proof}

\noindent
\textbf {Remark 4.}
As we mentioned in the Introduction, when $n=5$, we only know the Green's function of Paneitz operator has asymptotic expansion $G(p, x)= |x|^{4-n} + O(|x|^{4-n+\epsilon})$ as $|x|\rightarrow 0$, for some $\epsilon >0$. However, for any fixed $\epsilon>0$, if we let $a=b^{\frac{5}{\epsilon}}$ and $-\frac{\epsilon}{5}<\delta<0$, we can still prove Theorem \ref{Thm1} for $n=5$. However, when $n=5$, while Proposition \ref{p2}, \ref{p4}, \ref{p5} still hold, we are not aware of any good examples satisfying the assumption (i) in Theorem \ref{Thm1}.

\end{document}